%% file: RDPEnriquesVer3.tex
\documentclass{amsart}

\usepackage{amssymb}
\usepackage{amsmath}
\usepackage{lscape}
\usepackage{enumerate}
\usepackage{rotating}

\input mymacros
\input mymacrosEnrRs
\begin{document}
\title[Rational double points on Enriques surfaces]%
{Rational double points on Enriques surfaces}
\author{Ichiro Shimada}
\address{Department of Mathematics, 
Graduate School of Science, 
Hiroshima University,
1-3-1 Kagamiyama, 
Higashi-Hiroshima, 
739-8526 JAPAN}
\email{ichiro-shimada@hiroshima-u.ac.jp}
\thanks{This work was supported by JSPS KAKENHI Grant Number 16H03926 and 16K13749.}

\begin{abstract}
We classify, up to some lattice-theoretic equivalence,
all possible configurations of rational double points that 
can appear on a surface whose minimal resolution is  a complex Enriques surface.
\end{abstract}


\maketitle
\section{Introduction}
We work over the complex number field $\C$.
\par
The automorphism group of an Enriques surface
changes in a complicated way under specializations of 
the surface~(\cite[Section 3.1]{BarthPeters1983},~\cite[Remark 7.17]{ShimadaHessian}),
and smooth rational curves on the surface  control the change of the automorphism group.
Hence the study of configurations of smooth rational curves
are important for the explicit description  of 
variations of  automorphism groups on the family of Enriques surfaces (see Nikulin~\cite{Nikulin1984}).
On the other hand, 
in the investigation of $\Q$-homology projective planes,
Hwang-Keum-Ohashi~\cite{HKO2015} and Sch\"utt~\cite{Schutt2015}
classified all possible maximal root systems of smooth rational curves on Enriques surfaces.
\par
In this paper, we classify all root systems of smooth rational curves on Enriques surfaces
up to certain equivalence relation.
The list we obtain (see Theorem~\ref{thm:strongmain} and Table~\ref{table:main1}) includes,
of course,  the $31$ root systems 
of rank $9$ classified in~\cite{HKO2015} and~\cite{Schutt2015}.
Our method is purely lattice-theoretic and  algorithmic.
The main tool is the generalized 
Borcherds algorithm~(\cite{Borcherds1987},~\cite{Borcherds1998},~\cite{ShimadaIMRN}) 
to calculate the orthogonal group of a hyperbolic lattice.
An advantage of our method is that 
we can obtain the whole result by a single set of algorithms,
and that
we are exempted from the case-by-case investigation of possible root systems.
\par
\medskip
A lattice $L$ of rank $n$ is said to be \emph{hyperbolic} 
if the real quadratic space $L\tensor\R$ is of signature $(1, n-1)$.
Let $L$ be a hyperbolic lattice.
A \emph{positive cone} of $L$ is one of the two connected components
of $\shortset{x\in L\tensor\R}{\intf{x,x}>0}$.
Let $ \OG^+(L)$ denote the group of isometries of $L$ that preserve a positive cone.
\par
A \emph{root} of a lattice $L$ is a vector of square-norm $-2$.
We say that $L$ is a \emph{root lattice} if $L$ is generated by roots.
An \emph{$\ADE$-configuration of roots} of $L$ is a finite set $\Phi$ of roots of $L$
such that each connected component of the Dynkin diagram of $\Phi$ 
is of type $A_l$ ($l\ge 1$), $D_m$ ($m\ge 4$), or $E_n$ ($n=6,7,8$).
The \emph{$\ADE$-type $\tau(\Phi)$} of an $\ADE$-configuration $\Phi$ of roots
is the $\ADE$-type of  the Dynkin diagram of $\Phi$.
It is well-known that a negative-definite root lattice $R$ has an $\ADE$-configuration $\Phi_R$ of roots
that form a basis of $R$.
We define the $\ADE$-type $\tau(R)$ of $R$ to be $\tau(\Phi_R)$.
Conversely,
for  an $\ADE$-type  $t$,
we have a negative-definite root lattice $R(t)$,
unique up to isomorphism, 
such that $\tau(R(t))= t$.
\par
Let $\Lten$ be an even unimodular hyperbolic lattice of rank $10$,
which is unique up to isomorphism,
and which has a basis $E_{10}:=\{e_1, \dots, e_{10}\} $ 
consisting of roots whose Dynkin diagram is given in Figure~\ref{fig:E10}.
\begin{figure}
\def\ha{40}
\def\hav{37}
\def\hd{25}
\def\hdv{22}
\def\he{10}
\def\hev{7}
\setlength{\unitlength}{1.5mm}
{\small
\begin{picture}(80,11)(-5, 6)
\put(22, 16){\circle{1}}
\put(23.5, 15.5){$e\sb 1$}
\put(22, 10.5){\line(0,1){5}}
\put(9.5, \hev){$e\sb 2$}
\put(15.5, \hev){$e\sb 3$}
\put(21.5, \hev){$e\sb 4$}
\put(27.5, \hev){$e\sb 5$}
\put(33.5, \hev){$e\sb 6$}
\put(39.5, \hev){$e\sb 7$}
\put(45.5, \hev){$e\sb {8}$}
\put(51.5, \hev){$e\sb {9}$}
\put(57.5, \hev){$e\sb {10}$}
\put(10, \he){\circle{1}}
\put(16, \he){\circle{1}}
\put(22, \he){\circle{1}}
\put(28, \he){\circle{1}}
\put(34, \he){\circle{1}}
\put(40, \he){\circle{1}}
\put(46, \he){\circle{1}}
\put(52, \he){\circle{1}}
\put(58, \he){\circle{1}}
\put(10.5, \he){\line(5, 0){5}}
\put(16.5, \he){\line(5, 0){5}}
\put(22.5, \he){\line(5, 0){5}}
\put(28.5, \he){\line(5, 0){5}}
\put(34.5, \he){\line(5, 0){5}}
\put(40.5, \he){\line(5, 0){5}}
\put(46.5, \he){\line(5, 0){5}}
\put(52.5, \he){\line(5, 0){5}}
\end{picture}
}
\caption{Dynkin diagram of the basis $E_{10}$ of $\Lten$}\label{fig:E10}
\end{figure}
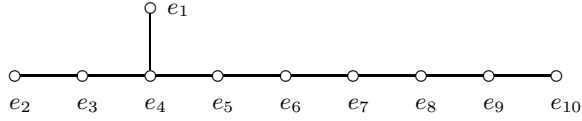
\par
For a smooth projective  surface $Z$, 
we denote by $S_Z$ the lattice of numerical equivalence classes of divisors of $Z$,
and by $\PPP_{Z}$ the positive cone of 
the hyperbolic lattice $S_Z$ containing an ample class.
Note that, if $Y$ is an Enriques surface, then  $\SY$ is isomorphic to $\Lten$.
\par
\medskip
An \emph{$\RDP$-Enriques surface}
is a pair $(Y, \rho)$ of an Enriques surface $Y$ and a birational morphism $\rho\colon Y\to \Ybar$
to a surface $\Ybar$ that has only rational double points as its singularities.
Let $(Y, \rho)$ be an $\RDP$-Enriques surface,
and let $C_1, \dots, C_n$ be the smooth rational curves contracted by $\rho$.
We denote  by $\Phi_{\rho}\subset \SY$ the set of classes of 
$C_1, \dots, C_n$.
Then  $\Phi_{\rho}$ is an $\ADE$-configuration of roots of $\SY$,
and $\tau(\Phi_{\rho})$ is the $\ADE$-type of the rational double points on $\Ybar$.
%
\begin{definition}
Two $\RDP$-Enriques surfaces 
$(Y, \rho)$ and $(Y\sprime, \rho\sprime)$
are said to be \emph{equivalent}
if there exists an isometry   $\SY\isom S_{Y\sprime}$ 
of lattices that maps $\PPP_{Y}$ to $\PPP_{Y\sprime}$ and 
$\Phi_{\rho}$ to $\Phi_{\rho\sprime}$ bijectively.
\end{definition}
In this paper, we classify all 
equivalence classes of $\RDP$-Enriques surfaces $(Y, \rho)$.
The $\ADE$-type $\tau(\Phi_{\rho})$
is a principal invariant of
the equivalence class,
but this invariant is not enough to distinguish  all the equivalence classes.
\par 
Our first main theorem  is the following purely lattice-theoretic result.
Let $\Phi_{f}$ be an  $\ADE$-configuration of roots in $\Lten$.
We denote by $R_f$ the root sublattice of $\Lten$ generated by  $\Phi_f$, and by
 $\primR_f$ the primitive closure of 
$R_f$  in $\Lten$.
\begin{definition}\label{def:L10equiv1}
Two $\ADE$-configurations of  roots $\Phi_{f}\subset \Lten$ 
and $\Phi_{f\sprime}\subset  \Lten$
are said to be \emph{equivalent} 
if there exists an isometry $g\in \OG^+(\Lten)$ that maps $\Phi_f$ to $\Phi_{f\sprime}$
bijectively.  
\end{definition}
\begin{theorem}\label{thm:L10main}
{\rm (1)}
For any  $\ADE$-configuration $\Phi_{f}$ of roots in $\Lten$, 
the lattice   $\primR_f$   is a root lattice.
Two $\ADE$-configurations of  roots $\Phi_{f}\subset \Lten$ 
and $\Phi_{f\sprime}\subset  \Lten$
are equivalent  if and only if
$\tau(\Phi_{f})=\tau(\Phi_{f\sprime})$ and  $\tau(\primR_f)=\tau(\primR_{f\sprime})$.
\par
{\rm (2)} Let $(t, \bar{t})$ be a pair of $\ADE$-types.
Then there exists an $\ADE$-configuration of  roots $\Phi_{f}\subset \Lten$
such that  $(\tau(\Phi_{f}), \tau(\primR_f))=({t}, \bar{t})$
if and only if the following hold;
\begin{enumerate}[{\rm (i)}]
\item $R({t})$ is of rank $< 10$, 
\item $R(\bar{t})$ is an even overlattice of $R({t})$, and 
\item the Dynkin diagram of $\bar{t}$ is a sub-diagram of the  Dynkin diagram 
 of $E_{10}$.
\end{enumerate}
\end{theorem}
There exist exactly $184$ equivalence classes of 
 $\ADE$-configurations of roots in $\Lten$.
They are given in Table~\ref{table:main1}.
In the case where $R_f=\primR_f$,
the item $\tau(\primR_f)$ is simply denoted by $\Phi$.
%
\begin{corollary}\label{cor:main1}
Let $(Y, \rho)$ be an $\RDP$-Enriques surface, and 
$R_{\rho}$  the sublattice of $\SY$ generated by the set $\Phi_{\rho}$ of  classes of smooth rational curves contracted by $\rho$.
Then 
the primitive closure $\primR_{\rho}$ of $R_{\rho}$
in $\SY$
 is a root lattice.
Two $\RDP$-Enriques surfaces $(Y, \rho)$ and $(Y\sprime, \rho\sprime)$ are equivalent  if and only if
$\tau(\Phi_{\rho})=\tau(\Phi_{\rho\sprime})$ and $\tau(\primR_{\rho})=\tau(\primR_{\rho\sprime})$.
\end{corollary}
%
%
\input MainTable
\par
Our next problem  is to determine all 
equivalence classes of 
 $\ADE$-configurations $\Phi_f\subset \Lten$
that can be realized as the $\ADE$-configuration $\Phi_{\rho}\subset\SY\cong\Lten$
associated with an $\RDP$-Enriques surface $(Y, \rho)$.
Let  $(Y, \rho)$ be an $\RDP$-Enriques surface,
and let $C_1, \dots, C_n$ be the smooth rational curves on $Y$ contracted by $\rho$,
so that $\Phi_{\rho}=\{[C_1], \dots, [C_n]\}$.
Let $\pi\colon X\to Y$ denote the universal covering of $Y$,
and $\enrinvol\colon X\to X$ the deck-transformation of the double covering $\pi$.
Then $\pi^*\colon \SY\to \SX$ is injective,
and the image $ \pi^*\SY$ is equal to the invariant sublattice in $\SX$
of the action  of  $\enrinvol$ in $\SX$.
The pull-back 
of $C_i$ by $\pi$ splits into the disjoint union 
 of two smooth rational curves 
$C_i\sprime$ and  $C_i\spprime$ on $X$.
We put 
$$
\Phi^{\sim}_{\rho}:=\{[C_1\sprime], [C_1\spprime], \dots, [C_n\sprime], [C_n\spprime]\}
\;\;\subset\;\; \SX,
$$
that is, $\Phi^{\sim}_{\rho}$ is the set of classes of smooth rational curves on $X$ contracted
by $\rho\circ\pi\colon X\to\Ybar$.
We denote by $M_{\rho}$ the sublattice of $\SX$ generated  by $ \pi^*\SY$ 
and $\Phi^{\sim}_{\rho}$.
Then the rank of $M_{\rho}$ is equal to $10+n$.
We then denote by $\primM_{\rho}$ the primitive closure of $M_{\rho}$ in $\SX$.
Note that the action of $\enrinvol$ on $S_X$ preserves the sublattices $M_{\rho}$ and $\primM_{\rho}$.
We put
$$
Q_{(Y, \rho)}:=\primM_{\rho}/ M_{\rho}.
$$
\begin{definition}\label{def:stronglyequivalent}
Let $(Y\sprime, \rho\sprime)$ be another $\RDP$-Enriques surface
with the universal covering $\pi\sprime\colon X\sprime\to Y\sprime$.
We say that $(Y, \rho)$ and $(Y\sprime, \rho\sprime)$ are
\emph{strongly equivalent} if there exists an isometry
$$
\mu\colon \primM_{\rho}\isom \primM_{\rho\sprime}
$$
with the following properties;
the isometry $\mu$ maps $\pi^*S_Y$ to $\pi^{\prime*} S_{Y\sprime}$
isomorphically, and the isometry $\mu_Y\colon S_Y\isom S_{Y\sprime}$ 
induced by $\mu$ maps $\PPP_{Y}$ to $\PPP_{Y\sprime}$
and $\Phi_{\rho}$ to $\Phi_{\rho\sprime}$ bijectively.
An isometry $\mu\colon \primM_{\rho}\isom \primM_{\rho\sprime}$  satisfying 
these conditions is called a \emph{strong-equivalence isometry}.
 \end{definition}
 It is obvious that,  if $(Y, \rho)$ and $(Y\sprime, \rho\sprime)$ are strongly equivalent,
 then they are  equivalent.
 It is also obvious that a strong-equivalence isometry $\primM_{\rho}\isom \primM_{\rho\sprime}$ is compatible
 with the actions of the Enriques involutions on $\primM_{\rho}$ and on $ \primM_{\rho\sprime}$.
The following lemma is proved in Section~\ref{subsec:latticesassociatedwith}.
\begin{lemma}\label{lem:QYrho}
Let $\mu\colon \primM_{\rho}\isom \primM_{\rho\sprime}$ be a strong-equivalence isometry.
Then  $\mu$ maps $\Phi^{\sim}_{\rho}$ to $\Phi^{\sim}_{\rho\sprime}$
bijectively.
In particular, 
if $(Y, \rho)$ and $(Y\sprime, \rho\sprime)$ are
strongly equivalent, then we have 
$Q_{(Y, \rho)}\cong Q_{(Y\sprime, \rho\sprime)}$.
\end{lemma}
Our second main result is as follows.
\begin{theorem}\label{thm:strongmain}
There exist exactly $265$ strong equivalence classes of $\RDP$-Enriques surfaces $(Y, \rho)$ with $\Phi_{\rho}\ne \emptyset$.
The invariants $\tau(\Phi_{\rho})$, $\tau(\primR_{\rho})$,
and $Q_{(Y, \rho)}$ are given
in Table~\ref{table:main1}.
\end{theorem}
In Table~\ref{table:main1},
the group $Q_{(Y, \rho)}$ is written in the following abbreviations:
$$
0=\{0\},
\;\;
n=\Z/n\Z\;\;(n\le 6),
\;\;
22=(\Z/2\Z)^2,
\;\;
222=(\Z/2\Z)^3,
\;\;
42=\Z/4\Z\times \Z/2\Z.
$$
When the fourth column $Q_{(Y, \rho)}$  is empty ($-$),
there exist no $\RDP$-Enriques surfaces $(Y, \rho)$
such that $\tau(\Phi_{\rho})$ and $\tau(\primR_\rho)$ are given in
the second and the third columns.
\begin{corollary}\label{cor:175}
There exist exactly $175$ equivalence classes of $\RDP$-Enriques surfaces
 $(Y, \rho)$ with $\Phi_{\rho}\ne \emptyset$.
\end{corollary}
\begin{example}
There exist no $\RDP$-Enriques surfaces $(Y, \rho)$ with singularities of type
 $\tau(\Phi_{\rho})=6A_1+A_2$.
\end{example}
\begin{example}
$\RDP$-Enriques surfaces $(Y, \rho)$ with singularities of type
 $\tau(\Phi_{\rho})=2A_1+2A_3$ are divided into $4$ equivalence classes 
 with
 $$
 \tau(\primR_\rho)= A_1 +E_7, \;\; A_3+D_5, \;\; D_8,\;\; E_8, 
 $$
 and into $3+2+7+4=16$ strong equivalence classes.
\end{example}
In~\cite{KeumZhang2002},
Keum and Zhang studied configurations $\AAA$ of 
smooth rational curves  on an Enriques surface $Y$
with $\ADE$-type
$cA_{p-1}$,
where $p$ is a prime,
and investigated 
the topological fundamental group $\pione (Y\setminus \AAA)$.
In~\cite{RamsSchuett2014},
Rams and Sch\"utt studied the case of $4A_2$,
and corrected the result of~\cite{KeumZhang2002}.
The relation of  the isomorphism classes of 
$\pione (Y\setminus \AAA)$
and our notion of strong equivalence relation
is not clear. 
We observe, however,  the following  fact 
by comparing 
our Table~\ref{table:main1}
with their results 
(Theorem~2 and Table~2 of~\cite{KeumZhang2002} and Theorem~4.3 of~\cite{RamsSchuett2014}):
Except for the case of the  $8A_1$ (and possibly for the case of the  $7A_1$),
the number of the isomorphism classes of 
$\pione (Y\setminus \AAA)$ and 
the number of strong equivalence classes coincide.
For example,  consider the case of $7A_1$.
The number of the  isomorphism classes of 
$\pione (Y\setminus \AAA)$ is two or three
(the authors of~\cite{KeumZhang2002} did not determine 
the realizability of the case  
$\pione (Y\setminus \AAA)\cong (\Z/2\Z)^4$),
whereas we have two strong equivalence classes
(Nos.~48 and~49).
Therefore we guess that the case $\pione (Y\setminus \AAA)\cong (\Z/2\Z)^4$
is not realizable.
For another example,  consider the case of $4A_2$.
Theorem~4.3 of~\cite{RamsSchuett2014}
says that there exist three  isomorphism classes of 
$\pione (Y\setminus \AAA)$,
two of which are
realized by  the same Enriques surface.
We also have $3=2+1$ strong equivalence classes
(Nos.~124 and~125).
For the case of  $8A_1$,
Table~2~of~\cite{KeumZhang2002} indicates 
only one 
isomorphism class of 
$\pione (Y\setminus \AAA)$,
whereas  we have two strong equivalence classes
(Nos.~87 and~88).

%
%
\par
An important class of involutions of  $K3$ surfaces 
other than Enriques involutions comes from the sextic double plane models.
The classification of $\ADE$-types of rational double points
on normal  sextic double planes
was given by Yang~\cite{Yang1996}.
A finer classification of  rational double points
on normal sextic double planes was given in~\cite{Shimada2010}
in the relation with the topology of Zariski pairs of plane curves (see~\cite{Artal1994},~\cite{Shimada2008AGEA}).
This classification was further refined  to 
the complete description of connected components 
of the equisingular families  of irreducible sextic plane curves with only simple singularities
by Akyol and Degtyarev~\cite{AkyolDegtyarev2015}.
The present article may be regarded as the Enriques  counterpart of
these studies of  plane sextic curves.
\par
This paper is organized as follows.
In Section~\ref{sec:preliminaries},
we review preliminary results about lattices.
Some algorithms for negative-definite root lattices are presented,
and the notion of chambers in a hyperbolic lattice 
is introduced.
In Section~\ref{sec:Lten}, 
we study the  lattice $\Lten$ in detail,
and prove Theorem~\ref{thm:L10main}
by the generalized Borcherds algorithm~(\cite{Borcherds1987},~\cite{Borcherds1998},~\cite{ShimadaIMRN}).
The classical result due to  Vinberg~\cite{Vinberg1975} plays an important role.
In Section~\ref{sec:geom},
we give a method to enumerate all the strong equivalence classes of $\RDP$-Enriques surfaces,
explain how to carry out this method, 
and prove Theorem~\ref{thm:strongmain}. 
%
\par
For the computation,
we used GAP~\cite{GAP}.
A computational data is available from the author's webpage~\cite{EnrRcompdata}.
In particular, the isomorphism class of the lattice $\primM_{\rho}$
for each strongly equivalence class of $\RDP$-Enriques surfaces
is given explicitly in~\cite{EnrRcompdata}.
Since we have $\primM_{\rho}=\SX$
when $(Y, \rho)$ corresponds to a general  point of 
an irreducible component of the moduli of $\RDP$-Enriques surfaces
with a fixed $\ADE$-type, 
this data will be useful in the study of the moduli  
of $\RDP$-Enriques surfaces.
\par
\medskip
The author thanks Professors Matthias~Sch\"utt and Hisanori~Ohashi 
for many discussions.
Thanks are also due to the referees 
for many helpful comments on the first version of the manuscript.
\section{Preliminaries}\label{sec:preliminaries}
%
%
\subsection{An $\ADE$-configuration}\label{subsec:ADEconfig}
A \emph{Dynkin configuration} is a finite set $\Phi$ with a map
$$
\intfvoid\colon \Phi\times \Phi\to \{-2,0,1\}
$$
such that $\intf{x, y}=\intf{y, x}$ for all $x, y\in \Phi$,
$\intf{x, x}=-2$ for all $x\in \Phi$, and
$\intf{x, y}\in \{0, 1\}$ for all $x,y\in \Phi$ with $x\ne y$.
With a Dynkin configuration $\Phi=\{r_1, \dots, r_n\}$,
we associate  its \emph{Dynkin diagram},
which is a graph whose set of vertices is $\Phi$ 
and whose set of edges is the set of pairs $\{r_i, r_j\}$ such that $\intf{r_i, r_j}=1$.
We say that a Dynkin configuration  is an \emph{$\ADE$-configuration}
if every connected component of its Dynkin diagram is of type
$A_{l}\; (l\ge 1)$, $D_m \;(m\ge 4)$, or $E_n\;  (n=6,7,8)$.
We define the \emph{$\ADE$-type} $\tau(\Phi)$ of an $\ADE$-configuration $\Phi$ 
to be the sum of the types of the connected components of its Dynkin diagram.
Let $\Phi$ and $\Phi\sprime$ be Dynkin configurations.
An isomorphism from $\Phi$  to $\Phi\sprime$ is a bijection $\Phi\isom \Phi\sprime$  that preserves 
$\intfvoid$.
The isomorphism class of an $\ADE$-configuration is uniquely determined by its $\ADE$-type.
We denote by $\Aut(\Phi)$ the group of automorphisms of $\Phi$,
which we let act on $\Phi$ from the \emph{right}.
%
%
%
\subsection{A lattice}\label{subsec:lattice}
%
Let $L$ be a free $\Z$-module of finite rank,
and $R$ a submodule of $L$.
The \emph{primitive closure} $\primR$ of $R$ in $L$ is 
the intersection of $R\tensor\Q$ and $L$ in $L\tensor\Q$.
We say that $R$ is \emph{primitive in $L$} if $R=\primR$.
%
%
%
\par
%
%

A \emph{lattice} is a free $\Z$-module $L$ of finite rank with a non-degenerate symmetric bilinear form
$$
\intfvoid\colon L\times L\to \Z.
$$
Let $L$ be a lattice.
We say that  $L$ is \emph{even} if $\intf{x, x}\in 2\Z$ for all $x\in L$.
The group of isometries of  $L$ is denoted by $\OG(L)$.
We let $\OG(L)$ act on $L$ from the \emph{right}.
An \emph{embedding} of a Dynkin configuration $\Phi$ into  $L$ is an injection $\Phi\inj L$
that preserves $\intfvoid$.
We define the \emph{dual lattice} $L\dual$ of $L$ by 
 $$
 L\dual:=\set{v\in L\tensor\Q}{\intf{x, v}\in \Z\;\; \textrm{for all}\;\; x\in L}.
 $$
The finite abelian group $\discg{L}:=L\dual/L$ is called the \emph{discriminant group} of $L$.
The group $\OG(L)$ acts on $\discg{L}$ from the right.
We say that  $L$ is \emph{unimodular} if $\discg{L}$ is trivial.
The \emph{signature} of  $L$ is the signature of the real quadratic space $L\tensor\R$.
Suppose that  $L$ is of rank $n>0$.
We say that  $L$  is \emph{hyperbolic} if the signature is $(1, n-1)$,
and is \emph{negative-definite} if the signature is $(0, n)$.
%
%
\subsection{Roots and reflections}\label{subsec:roots}
%
Let $L$ be an even lattice.
A \emph{root} of $L$ is a vector $r\in L$ such that $\intf{r, r}=-2$.
The set of roots of $L$ is denoted by  $\Roots (L)$.
A root $r$ of $L$ defines an isometry
$$
\refl{r}\colon x\mapsto x+\intf{x, r}r
$$
of $L$, which is called the \emph{reflection} associated with  $r$.
We denote by $W(L)$ the subgroup of $\OG(L)$ generated by
all the reflections associated with the roots, 
and call it the \emph{Weyl group} of $L$.
We say that $L$ is a \emph{root lattice} if $L$ is generated by roots.
Let $\Phi$ be a subset of $\Roots(L)$.
We denote by $W(\Phi, L)$
the subgroup of $W(L)$ generated by all the reflections $s_r$  associated with  $r\in \Phi$.
%
%
%
\subsection{A negative-definite root lattice}\label{subsec:negdefroot}
%
%
Let $\Phi$ be an $\ADE$-configuration.
Extending $\intfvoid\colon\Phi\times\Phi\to \Z$ by linearity 
to the bilinear form on the free $\Z$-module generated by $\Phi$,
we obtain  a  negative-definite root lattice  of rank $|\Phi|$,
which we will denote by $\gen{\Phi}$.
Conversely,
let $R$ be a negative-definite root lattice.
Then  $R$
has a basis $\Phi_R$ consisting of roots 
that form an $\ADE$-configuration,
which we call an \emph{$\ADE$-basis} of $R$.
We define the $\ADE$-type $\tau(R)$ of  $R$
to be  the $\ADE$-type $\tau(\Phi_R)$ of an $\ADE$-basis $\Phi_R$ of $R$.
In the following, 
we describe the set of all $\ADE$-bases of a negative-definite root lattice.
See~\cite[Chapter~1]{EbelingBook} or~\cite[Chapter~1]{HumphreysBook} for the proof.
\par
Let $R$ be a negative-definite root lattice.
For a root $r$ of $R$, we denote by  $r\sperp$ the hyperplane of $R\tensor\R$
defined by $\intf{x, r}=0$.
We then put
$$
(R\tensor\R)\spcirc:=(R\tensor\R)\;\setminus\; \bigcup\, r\sperp,
$$
where $r$ runs through the finite set $\Roots(R)$. 
Let $\Gamma$ be a connected component of $(R\tensor\R)\spcirc$,
and $\closure{\Gamma}$ the closure of $\Gamma$ in $R\tensor\R$.
Then the set $\Phi_{\Gamma}$ of all $r\in \Roots(R)$ such that
$\intf{x, r}>0$ for any $x\in \Gamma$ and that
$r\sperp\cap \closure{\Gamma}$ contains a non-empty open subset of $r\sperp$
form an $\ADE$-basis of $R$,
and the mapping $\Gamma\mapsto \Phi_{\Gamma}$ gives a bijection from the set of connected components 
of $(R\tensor\R)\spcirc$ to the set of $\ADE$-bases of $R$.
For $x\in (R\tensor\R)\spcirc$, we denote by 
$\Gamma(x)$ the connected component of $(R\tensor\R)\spcirc$ containing $x$.
Let  $\Phi=\{r_1, \dots, r_n\}$ be an $\ADE$-basis of $R$.
We put
$$
c:=r_1\dual+\cdots+r_n\dual\;\;\in\;\;(R\tensor\R)\spcirc,
$$
where $r_1\dual, \dots, r_n\dual$ are the basis of $R\dual$ dual to 
the basis $r_1, \dots, r_n$ of $R$.
Then the connected component of $(R\tensor\R)\spcirc$ 
corresponding to the $\ADE$-basis $\Phi$ is $\Gamma(c)$.
It is obvious that we have a natural embedding
$\Aut(\Phi)\inj \OG(R)$ whose image is
$$
\Stab (\Gamma(c), R):=\set{g\in \OG(R)}{\Gamma(c)^g=\Gamma(c)}=\set{g\in \OG(R)}{c^g=c}.
$$
Since $W(R)$ acts on the set of connected components of $(R\tensor\R)\spcirc$ simple-transitively,
we have a \emph{splitting}  exact sequence
\begin{equation}\label{eq:WRORkappa}
1\;\maprightsp{}\; W(R)\;\maprightsp{}\; \OG(R)\;\maprightsp{\kappa}\; \Aut(\Phi)\;\maprightsp{} \; 1.
\end{equation}
\begin{algorithm}\label{algo:Gammauv}
Let $u$ and $v$ be points of $(R\tensor \Q)\cap (R\tensor\R)\spcirc$.
This algorithm finds the unique element $g\in W(R)$
that maps $\Gamma(u)$ to $\Gamma(v)$.
Let $\xi$ be a sufficiently general element of $R\tensor\Q$,
and let $\varepsilon$ be a sufficiently small positive rational number.
We consider the open  line segment in $R\tensor\R$ drawn by the point
$$
p(t):=u+ t (v+\varepsilon \xi), 
$$
where $t$ moves in the set of positive real numbers.
Let $\{r_1, \dots, r_N\}$ be the set of roots $r_i$ of $R$
such that $\intf{u,r_i}<0$ and  $\intf{v,r_i}>0$.
For each $r_i$,
let  $t_i$ be the unique rational number such that
$\intf{p(t_i), r_i}=0$.
Since the perturbation vector $\xi$ is general,
we can assume that $t_1, \dots, t_N$ are distinct.
We sort $r_1, \dots, r_N$ in such a way that $t_1<\dots<t_N$.
Then $g:=s_{r_1}\dots s_{r_N}\in W(R)$ satisfies $\Gamma(u)^g=\Gamma(v)$.
\end{algorithm}
As applications of Algorithm~\ref{algo:Gammauv},
we obtain the following algorithms.
As noted above,  the stabilizer subgroup $\Stab (\Gamma(c), R)$  is 
canonically  identified with $\Aut(\Phi)$.
\begin{algorithm}\label{algo:phi}
Let an isometry $g\in \OG(R)$ be given.
The image $\kappa (g)\in \Aut(\Phi)$ of $g$ by the homomorphism
$\kappa$ in~\eqref{eq:WRORkappa}  is calculated by 
applying Algorithm~\ref{algo:Gammauv} to $u=c^g$ and $v=c$.
We find $h\in W(R)$ such that $\Gamma(c)^{gh}=\Gamma(c)$.
Hence we have $\kappa(g)=gh$.
\end{algorithm}
\begin{algorithm}\label{algo:oppR}
Applying Algorithm~\ref{algo:Gammauv} to $u=c$ and $v=-c$,
we find the element $l\in W(R)$ such that $\Gamma(c)^{l}=-\Gamma(c)$.
This element $l$ is the \emph{longest element} of the Coxeter group $W(R)$. 
(See~\cite[Section~1.8]{HumphreysBook}.)
\end{algorithm}
\par
Another method to obtain an $\ADE$-basis of $R$ is as follows.
We put
$$
\Hom(R, \R)\spcirc:=\set{\ell \in \Hom(R, \R)}{\ell (r)\ne 0\;\;\textrm{ for any $r\in \Roots(R)$}}, 
$$
and for $\ell \in \Hom(R, \R)\spcirc$, we put $\Roots(R)_{\ell>0}:=\shortset{r\in \Roots(R)}{\ell(r)>0}$.
\begin{definition}\label{def:indecomp}
Let $S$ be a subset of $\Roots(R)_{\ell>0}$.
We say that $r\in S$ is \emph{indecomposable in $S$}
if $r$ is not written as a linear combination $\sum a_i r_i$
of elements $r_i\in S$ with $a_i\in \Z_{\ge 0}$ such that $\sum a_i>1$.
\end{definition}
Let  $\Phi_{\ell>0}$  be  the set of roots $r\in \Roots(R)_{\ell>0}$ indecomposable in $\Roots(R)_{\ell>0}$.
Then  $\Phi_{\ell>0}$ 
is an $\ADE$-basis  of $R$, and the mapping $\ell\mapsto \Phi_{\ell>0}$
 gives a bijection from the set of connected components 
of $\Hom(R, \R)\spcirc$ to the set of $\ADE$-bases of $R$.
This correspondence $\ell\mapsto \Phi_{\ell>0}$
will be used in Section~\ref{sec:geom}.
%
%
\subsection{An even hyperbolic  lattice}\label{subsec:hyperbolic}
%
%
Let $L$ be an even hyperbolic lattice.
A \emph{positive cone} is one of the two  connected components
of the space $\shortset{x\in L\tensor\R}{\intf{x, x}>0}$.
Let $\PPP$ be a positive cone.
We denote by $\OG^+(L)$ the stabilizer subgroup of $\PPP$ in $\OG(L)$.
For a non-zero vector $v\in L\tensor\R$, we put
$$
H^+_v:=\set{x\in \PPP}{\intf{v, x}\ge 0},
\quad  
(v)\sperp:=\set{x\in \PPP}{\intf{v, x}=0}.
$$
Then $(v)\sperp\ne \emptyset$ if and only if $\intf{v,v}<0$.
Note that the Weyl group $W(L)$ acts on $\PPP$.
A \emph{standard fundamental domain of the action  of $W(L)$ on $\PPP$} is 
the closure in $\PPP$ of a connected component of
$$
\PPP\;\;\setminus\;\; \bigcup\; (r)\sperp,
$$
where $r$ runs through $\Roots(L)$.
Then $W(L)$ acts on the set of standard fundamental domains simple-transitively.
Let $\Delta$ be a standard fundamental domain.
We put
$\Aut(\Delta):=\shortset{g\in \OG^+(L)}{\Delta^g=\Delta}$.
Then we have a \emph{splitting} exact sequence
\begin{equation}\label{eq:splithyp}
1\;\maprightsp{}\; W(L)\;\maprightsp{}\; \OG^+(L)\;\maprightsp{}\; \Aut(\Delta)\;\maprightsp{}\; 1,
\end{equation}
and we have a tessellation
\begin{equation}
\PPP=\bigcup_{g\in W(L)}\; \Delta^{g}.
\label{eq:tess}
\end{equation}
%
%
%
\subsection{Chambers}\label{subsec:chambers}
%
%
Let $L$ be an even hyperbolic lattice, and $\PPP$ a positive cone of $L$.
A closed subset $\CCC$ of $\PPP$ is called a \emph{chamber} if 
the interior $\CCC\spcirc$ of $\CCC$ in $\PPP$ is non-empty
and there exists a set of vectors $\HHH\subset L\tensor\R$  
such that the family of hyperplanes $\shortset{(v)\sperp}{v\in \HHH}$ of $\PPP$
is locally finite in $\PPP$ and that 
\begin{equation}\label{eq:CCC}
\CCC=\bigcap_{v\in \HHH} H^+_v.
\end{equation}
Let $\CCC$ be a chamber defined by a set of vectors $\HHH\subset L\tensor\R$ as in~\eqref{eq:CCC}.
A  closed subset $F$ of $\CCC$ is called a \emph{face} of $\CCC$ if
$F\ne \emptyset$, $F\cap \CCC\spcirc=\emptyset$,  and 
there exists a subset $\HHH_F$ of $\HHH$ such that
$$
F=\CCC\cap  \bigcap_{v\in \HHH_F} (v)\sperp.
$$
If $F$ is a face, then we have a  unique  linear subspace $V$ of $L\tensor\R$
such that $V\cap F$ contains a non-empty open subset of $V$.
We say that $V$ or $\PPP\cap V$  the \emph{supporting linear subspace} of $F$.
The \emph{codimension} of a face is defined to be the codimension of 
its supporting linear subspace in $L\tensor\R$ or in $\PPP$.
A face of codimension $1$ is called a \emph{wall}.
Let $F$ be a wall of $\CCC$.
A vector $v\in L\tensor\R$ is said to \emph{define the wall} $F$ if 
 $\CCC$ is contained in $H^+_{v}$ and $F= \CCC\cap (v)\sperp$ holds.
%
%
\subsection{The discriminant form and overlattices}\label{subsec:discform}
Let $L$ be an even lattice.
Recall that $\discg{L}:=L\dual/L$.
Then the natural $\Q$-valued symmetric bilinear form on $L\dual$ defines
a finite quadratic form
$$
\discf{L}\colon \discg{L}\to \Q/2\Z,
$$
which we call the \emph{discriminant form} of $L$.
See Nikulin~\cite{Nikulin1979} for the basic properties of the discriminant form.
We have a natural homomorphism
$\OG(L)\to \OG(\discf{L})$,
where $\OG(\discf{L})$ is 
the automorphism group  of the finite quadratic form $\discf{L}$.
An \emph{even overlattice} of $L$ is a submodule $M$ of $L\dual$ containing $L$ such that 
the restriction of the natural $\Q$-valued symmetric bilinear form on $L\dual$ makes $M$ an even lattice.
By definition, we have the following:
\begin{proposition}\label{prop:overlattices}
The map $M\mapsto M/L$ gives a bijection from the set of even overlattices $M$ of $L$
to the set of  totally isotropic subgroups of $\discf{L}$.
\qed
\end{proposition}
%
Note that $\OG(L)$ acts on the set of even overlattices of $L$ from the right.
\begin{proposition}\label{prop:genus}
Suppose that the signature of $L$ is $(s_+, s_-)$.
Let $H$ be an even unimodular lattice of signature $(h_+, h_-)$.
Then $L$ can be embedded into $H$ primitively if and only if
there exists an even lattice of signature
$(h_+-s_+, h_- - s_-)$ whose discriminant form is isomorphic to
$-\discf{L}$.
\qed
\end{proposition}
\begin{remark}\label{rem:genus}
The signature $(h_+-s_+, h_- - s_-)$
and the isomorphism class of the discriminant form $-\discf{L}$ 
determine a genus of even lattices.
There exist various versions of  the criterion to determine
whether a genus given by  signature and a finite quadratic form is empty or not.
See, for example, Nikulin~\cite{Nikulin1979}, Conway-Sloane~\cite[Chapter~15]{ConwaySloaneBook}, Miranda-Morrison~\cite{MirandaMorrisonBook}.
This criterion has been applied to many problems of $K3$ surfaces.
See, for example,~\cite{AkyolDegtyarev2015,  
Shimada2007, Shimada2010, ShimadaConnected,
Yang1996}.
\end{remark}
%
%
%
%
%
\section{The lattice $\Lten$}\label{sec:Lten}  %
%
For an embedding $f\colon \Phi\inj\Lten$ of an $\ADE$-configuration $\Phi$ into $\Lten$,
we denote by $\Phi_f$ the image of $\Phi$ by $f$.
We extend  Definition~\ref{def:L10equiv1} to the equivalence relation of
embeddings of $\ADE$-configurations into $\Lten$.
%
\begin{definition}\label{def:equivL10}
Let $f\colon \Phi\inj\Lten$  and  $f\sprime\colon \Phi\sprime\inj \Lten$ be
embeddings 
of  $\ADE$-configurations $\Phi$ and $\Phi\sprime$.
We say that $f$ and $f\sprime$ 
are \emph{equivalent} 
if there exist an isomorphism  $\gamma\colon \Phi\isom \Phi\sprime$
and an isometry $g\in \OG^+(\Lten)$ that make the following diagram commutative:
$$
\begin{array}{ccc}
\Phi &\maprightinjsp{f} & \Lten\mystruth{12pt} \\
\llap{\scriptsize $\gamma$\;}\mapdown && \mapdown{\rlap{\scriptsize $g$}}\\
\Phi\sprime  &\maprightinjsp{f\sprime} & \Lten\rlap{.}
\end{array}
$$
\end{definition}
The purpose of this section is to prove Theorem~\ref{thm:L10main}.
\subsection{Negative-definite primitive root sublattices of $\Lten$}\label{subsec:primrootsublattices}
\begin{definition}
Let $\NNN$ denote the set of all negative-definite primitive root sublattices of $\Lten$,
on which $\OG^+(\Lten)$ acts from the right.
\end{definition}
We calculate the set $\NNN/\OG^+(\Lten)$  of orbits of this action. 
Recall that the lattice $\Lten$ has a basis 
$E_{10}=\{e_1, \dots, e_{10}\}$
consisting of roots that form the Dynkin diagram in Figure~\ref{fig:E10}.
We fix, once and for all,
 the positive cone $\Pten$ that contains the vector 
\begin{equation*}\label{eq:w0}
c_0:=e_1\dual+ \dots + e_{10}\dual
\end{equation*}
of square-norm $1240$,
where $e_1\dual, \dots, e_{10}\dual$ are the basis of $\Lten\dual=\Lten$
dual to the basis $e_1, \dots, e_{10}$.
For simplicity, 
we put,  for a subset $S$ of $\Lten\tensor\R$,
\begin{eqnarray}
[S]\sperp &:=& \set{v\in \Lten}{\intf{v, x}=0\;\;\textrm{for all}\;\; x\in S}, \label{eq:Sperp1}\\
(S)\sperp &:=& \set{v\in \Pten}{\intf{v, x}=0\;\;\textrm{for all}\;\; x\in S}. \label{eq:Sperp2}
\end{eqnarray}
We consider the chamber
$$
\Delta_0:=\set{x\in \Pten}{\intf{x, e_i}\ge 0\;\;\textrm{for all}\;\; i=1. \dots, 10},
$$
which contains $c_0$ in its interior.
Vinberg~\cite{Vinberg1975} proved the following:
\begin{theorem}[Vinberg]
Each $e_i$ defines a wall of $\Delta_0$.
The chamber $\Delta_0$ is a standard fundamental domain of the action of $W(\Lten)$ on $\Pten$.
\qed
\end{theorem}
\begin{definition}
We call a standard fundamental domain of the action of $W(\Lten)$ on $\Pten$
a \emph{Vinberg chamber}.
We denote by $\VVV$ the set of all Vinberg chambers.
\end{definition}
The following easy fact is used frequently in this section:
%
Let $r_1, \dots, r_n$ be roots of $\Lten$
such that the linear subspace $\PPP\sprime:=(\{r_1, \dots, r_n\})\sperp$ of $\PPP_{10}$ 
is non-empty.
Let $\Delta$ be a Vinberg chamber.
If $\PPP\sprime\cap \Delta$ contains a non-empty open subset of $\PPP\sprime$,
then $\PPP\sprime\cap \Delta$ is a face of $\Delta$,
and  its supporting linear subspace  is $\PPP\sprime$.
\par
\medskip
Since the Dynkin diagram in Figure~\ref{fig:E10} has no symmetries,
we see from~\eqref{eq:splithyp} that $\OG^+(\Lten)$ is equal to $W(\Lten)$. 
In particular, we have the following:
\begin{proposition}
The map $g\mapsto \Delta_0^{g}$ is a bijection 
from $\OG^+(\Lten)$ to $\VVV$.
\qed
\end{proposition}
We denote by $\gamma\colon \VVV\to \OG^+(\Lten)$ the inverse  map of $g\mapsto \Delta_0^{g}$, that is, $\gamma(\Delta)$
is the unique element of $\OG^+(\Lten)$ such that 
$$
\Delta_0^{\gamma(\Delta)}=\Delta.
$$
The following lemma is easy to prove:
\begin{lemma}
{\rm (1)} 
A subset $\Sigma$ of $E_{10}$ is an $\ADE$-configuration of roots if and only if
$\Sigma\ne\emptyset$, $\Sigma\ne E_{10}$, and $\Sigma\ne \{e_1, \dots, e_{9}\}$.
{\rm (2)}
Let $\closure{\PPP}_{10}$ and $\closure{\Delta}_{0}$ denote  the closure of $\Pten$ and $\Delta_0$ in $\Lten\tensor\R$,
respectively.
Then $\closure{\Delta}_{0}\cap (\closure{\PPP}_{10}\setminus \Pten)$ is equal to the half-line
$\R_{\ge 0}e_{10}\dual=( [\{e_1, \dots, e_{9}\}]\sperp \tensor\R)\cap \closure{\PPP}_{10}$.
\qed
\end{lemma}
Let $2^{E_{10}}$ be the power set of $E_{10}=\{e_1, \dots, e_{10}\}$.
We consider the set 
$$
\Sigmas:=2^{E_{10}}\setminus \{\emptyset, E_{10},  \{e_1, \dots, e_{9}\}\}.
$$
Since $\closure{\Delta}_{0}$  is the cone over a $9$-dimensional simplex, we see that
$$
\Sigma\;\;\mapsto\;\; F_{\Sigma}:=(\Sigma)\sperp\, \cap \, \Delta_0
$$
is a bijection from $\Sigmas$ to the set of faces of $\Delta_0$.
For an element $\Sigma$ of $\Sigmas$,
we denote by $\gen{\Sigma}$ the sublattice of $\Lten$ generated by $\Sigma$.
Since $\Sigma$ is a subset of the basis $E_{10}$ of $\Lten$,
the sublattice $\gen{\Sigma}$ is primitive in $\Lten$.
Therefore we obtain the following:
\begin{lemma}\label{lem:genSigma}
Let $\Sigma$ be an element of $\Sigmas$.
Then  $\gen{\Sigma} \in \NNN$ and $\gen{\Sigma}=[F_{\Sigma}]\sperp$.
\qed
\end{lemma}
%
%
\begin{lemma}\label{lem:primR}
Let $R$ be a negative-definite root sublattice of $\Lten$,
and $\primR$ the primitive closure of $R$ in $\Lten$.
Then there exists an isometry $g\in \OG^+(\Lten)$ 
such that $\primR^g=\gen{\Sigma}$
for some  $\Sigma\in \Sigmas$.
In particular, the lattice $\primR$ is a root lattice.
\end{lemma}
\begin{proof}
Suppose that $R$ is generated by roots $r_1, \dots, r_n$.
Then 
$$
\PPP(R):=(\{r_1, \dots, r_n\})\sperp=([R]\sperp\tensor \R )\cap \Pten
$$
is non-empty, 
because $\PPP(R)$ is a positive cone of the hyperbolic lattice $[R]\sperp$.
Since $\PPP(R)$ is contained in a hyperplane $(r_1)\sperp$,
we see that $\PPP(R)$ is disjoint from the interior of any Vinberg chamber.
Since $\Pten$ is tessellated by Vinberg chambers,
there exists a Vinberg chamber 
$\Delta$ such that $\PPP(R)\cap \Delta$ contains a non-empty open subset of $\PPP(R)$.
Therefore $\PPP(R)\cap \Delta$  is a face of $\Delta$.
Then $\gamma(\Delta)\inv\in \OG^+(\Lten)$ maps $\PPP(R)\cap \Delta$ to a face $F_{\Sigma}$ of $\Delta_0$
associated with some $\Sigma\in \SSS$.
Hence $\gamma(\Delta)\inv$ maps $\primR=[\PPP(R)\cap \Delta]\sperp$ to  $\gen{\Sigma}=[F_{\Sigma}]\sperp$.
\end{proof}
%
%
\begin{corollary}
The map $\Sigma\mapsto \gen{\Sigma}$ induces a surjection from $\Sigmas$ to $\NNN/\OG^+(\Lten)$.
\qed
\end{corollary}
For $\Sigma\in \Sigmas$, we put 
$$
\PPP(\Sigma):=(\Sigma)\sperp=([\Sigma]\sperp\tensor \R )\cap \Pten.
$$
\begin{definition}
For $\Sigma, \Sigma\sprime\in \Sigmas$,
we write $\Sigma\sim \Sigma\sprime$
if there exists an isometry $g\in \OG^+(\Lten)$
such that $\gen{\Sigma}^g=\gen{\Sigma\sprime}$, or equivalently, $\PPP(\Sigma)^g=\PPP(\Sigma\sprime)$.
\end{definition}
In the following, 
we fix an element $\Sigma\in \Sigmas$,
and calculate the set $\shortset{\Sigma\sprime\in \Sigmas}{\Sigma\sprime\sim\Sigma}$
and a finite generating set of the stabilizer subgroup 
$$
\Stab(\gen{\Sigma}, \Lten):=\shortset{g\in \OG^+(\Lten)}{\gen{\Sigma}^g=\gen{\Sigma}}=\shortset{g\in \OG^+(\Lten)}{\PPP(\Sigma)^g=
\PPP(\Sigma)}
$$ 
of $\gen{\Sigma}$ in $\OG^+(\Lten)$.
We put
$$
\VVVS:=\set{\Delta\in \VVV}{\textrm{$\PPP(\Sigma)\cap\Delta$ contains a non-empty open subset of $\PPP(\Sigma)$}}.
$$
\begin{definition}\label{def:induced}
If $\Delta\in \VVVS$, 
then $\PPP(\Sigma)\cap\Delta$ is a chamber in the positive cone $\PPP(\Sigma)$ of 
the hyperbolic lattice $[\Sigma]\sperp$.
A closed subset $D$ of $\PPP(\Sigma)$ is said to be an \emph{induced chamber} if $D$ is written as $\PPP(\Sigma)\cap\Delta$ 
by some $\Delta\in \VVVS$.
\end{definition}
By definition, the positive cone  $\PPP(\Sigma)$ of the hyperbolic lattice $[\Sigma]\sperp$ is covered by induced chambers 
in such a way that, if $D$ and $D\sprime$ are distinct induced chambers, 
then their interiors in $\PPP(\Sigma)$ are disjoint.
Let $\Delta$ be an arbitrary element of $\VVVS$.
Looking at the tessellation of $\Pten$ by Vinberg chambers
locally around an interior point of the induced chamber  $\PPP(\Sigma)\cap\Delta$,   
we obtain the following:
\begin{lemma}\label{lem:WSigma}
Let $W(\Sigma, \Lten)$ denote  the subgroup of $\OG^+(\Lten)$ generated by
the reflections $s_r$ associated with the roots $r\in \Sigma$.
Let $\Delta$ be an element of $\VVVS$.
Then $W(\Sigma, \Lten)$ acts on the set
$\shortset{\Delta\sprime\in\VVVS}{\PPP(\Sigma)\cap\Delta=\PPP(\Sigma)\cap\Delta\sprime}$
simple-transitively.
\qed
\end{lemma}
Note that,
if $\Delta\in \VVVS$,
then $\gamma(\Delta)\inv$ maps the induced chamber $\PPP(\Sigma)\cap\Delta$ of $\PPP(\Sigma)$ 
to a face of $\Delta_0$.
Hence we can define  a mapping $\sigma\colon \VVVS\to \Sigmas$ by the property
\begin{equation}\label{eq:defsigma}
(\PPP(\sigma(\Delta))\cap \Delta_0)^{\gamma(\Delta)}=\PPP(\Sigma)\cap\Delta,
\end{equation}
or equivalently, by the equality ${F_{\sigma(\Delta)}}^{\gamma(\Delta)}=\PPP(\Sigma)\cap\Delta$.
Taking the orthogonal complement of the both sides,
we obtain
\begin{equation}\label{eq:defsigma2}
\gen{\sigma(\Delta)}^{\gamma(\Delta)}=\gen{\Sigma}
\end{equation}
\begin{lemma}\label{lem:sigmaVVVS}
We have
 $\shortset{\Sigma\sprime\in\Sigmas}{\Sigma\sprime\sim\Sigma}=\shortset{\sigma(\Delta)}{\Delta\in\VVVS}$.
\end{lemma}
\begin{proof}
By~\eqref{eq:defsigma2},
we have $\sigma(\Delta)\sim \Sigma$
for any $\Delta\in \VVVS$.
Suppose that $\Sigma\sprime\sim\Sigma$,
and let $h\in \OG^+(\Lten)$ be an isometry such that $\PPP(\Sigma\sprime)^h=\PPP(\Sigma)$.
We put $\Delta:=\Delta_0^h$ so that $\gamma(\Delta)=h$.
Since  $(\PPP(\Sigma\sprime)\cap\Delta_0)^h=\PPP(\Sigma)\cap \Delta$
contains a non-empty open subset of $\PPP(\Sigma)$, we have $\Delta\in \VVVS$
and  $\Sigma\sprime=\sigma(\Delta)$ by the defining property~\eqref{eq:defsigma} of $\sigma$.
\end{proof}
Using Lemma~\ref{lem:WSigma}, we can prove the following:
\begin{lemma}\label{lem:sigmaWSigma}
For $\Delta, \Delta\sprime\in \VVVS$, 
if $\PPP(\Sigma)\cap\Delta=\PPP(\Sigma)\cap\Delta\sprime$,
then $\sigma(\Delta)=\sigma(\Delta\sprime)$.
\qed
\end{lemma}
%
Let $\Delta$ be an element of $\VVV_{\Sigma}$, and let $D:=\PPP(\Sigma)\cap\Delta$ be the corresponding induced chamber of $\PPP(\Sigma)$.
Looking at the isomorphism 
$$ 
F_{\sigma(\Delta)}=(\PPP(\sigma(\Delta))\cap \Delta_0)\isom \PPP(\Sigma)\cap\Delta=D
$$
induced by $\gamma(\Delta)\in \OG^+(\Lten)$, 
we see that the set of walls of   $D$  is equal to
$$
\set{{F_{\Xi}}^{\gamma(\Delta)}}{\textrm{$\Xi\in \Sigmas$ satisfies $\Sigma\subset \Xi$ and $|\Sigma|+1=|\Xi|$}}.
$$
\begin{definition}\label{def:adjacent}
Let $w:={F_{\Xi}}^{\gamma(\Delta)}$ be a wall of 
the induced chamber $D$ of $\PPP(\Sigma)$.
An induced chamber \emph{adjacent to $D$ across the wall $w$} is the unique 
 induced chamber $D\sprime$  of $\PPP(\Sigma)$ such that 
$D\ne D\sprime$  and that $w$ is a wall of $D\sprime$.
We say that an induced chamber $D\sprime$ is \emph{adjacent to $D$}
if $D\sprime$ is adjacent to $D$ across some  wall of $D$.
\end{definition}
%
%
\begin{algorithm}\label{algo:adjacent}
Let $D$ and $w$ be as in Definition~\ref{def:adjacent}.
By the following method, 
we can obtain a Vinberg chamber  $\Delta\sprime\in \VVVS$ and 
an isometry $\gamma(\Delta\sprime)\in \OG^+(\Lten)$ such that $\PPP(\Sigma)\cap\Delta\sprime$ is 
the  induced chamber adjacent to $D$ across  $w$. 
Let $\Xi$ be the element of $\Sigmas$ such that $w={F_{\Xi}}^{\gamma(\Delta)}$, 
and let $W(\Xi, \Lten)$ denote the subgroup of $\OG^+(\Lten)$ generated by the reflections $s_r$
associated with  the roots $r\in \Xi$.
Let
$\xi$ be the longest element of the Coxeter group $W(\Xi, \Lten)\cong W(\gen{\Xi})$,
which can be calculated by~Algorithm~\ref{algo:oppR}.
Then, if $\UUU$ is a sufficiently small neighborhood in $\Pten$ of a general point of 
the face  $F_{\Xi}$  of $\Delta_0$, we have
$$
\Delta_0^{\xi}\cap \UUU=\set{x\in \UUU}{\intf{x, r}\le 0\;\;\textrm{for all}\;\; r\in \Xi},
$$
that is, the Vinberg chamber $\Delta_0^{\xi}$ is opposite to $\Delta_0$ with respect to 
the face $F_{\Xi}$.
Then $\Delta\sprime:=\Delta_0^{\xi\gamma(\Delta)}$
is the desired Vinberg chamber, and 
we have $\gamma(\Delta\sprime)=\xi\gamma(\Delta)$.
\end{algorithm}
We present the main algorithm of this section,
which is based on the generalized Borcherds method~(\cite{Borcherds1987},~\cite{Borcherds1998},~\cite{ShimadaIMRN}).
\begin{algorithm}\label{algo:main}
Let an element $\Sigma$ of $\Sigmas$ be given.
This algorithm calculates the set $\shortset{\Sigma\sprime\in \Sigmas}{\Sigma\sprime\sim\Sigma}$
and a generating set  of the stabilizer subgroup $\Stab(\gen{\Sigma}, \Lten)$ of $\gen{\Sigma}$ in $\OG^+(\Lten)$.
We set
$$
\DDD:=[\Delta_0],
\quad
\gamma(\DDD):=[\id], 
\quad
\sigma(\DDD):=[\Sigma],
\quad
\GGG:=\{\},
\quad
i:=0,
$$
where $\id$ is the identity of $\OG^+(\Lten)$.
During the calculation,
we have the following:
\begin{enumerate}[(i)]
\item $\DDD$ is a list $[\Delta_0, \dots, \Delta_j]$ of elements of $\VVVS$
such that $\sigma(\Delta_{\mu})\ne \sigma(\Delta_{\nu})$ if $\mu\ne \nu$,
where $\sigma\colon \VVV_{\Sigma}\to \Sigmas$ is defined by~\eqref{eq:defsigma}, 
\item $\gamma(\DDD)$ is the list $[\gamma(\Delta_0), \dots, \gamma(\Delta_j)]$ of elements of $\OG^+(\Lten)$,
\item $\sigma(\DDD)$ is the list $[\sigma(\Delta_0), \dots, \sigma(\Delta_j)]$  of distinct elements  of $\SSS$, and
\item $\GGG$ is a set of elements of $\Stab(\gen{\Sigma}, \Lten)$.
\end{enumerate}
While $i+1\le j+1=|\DDD|$, we execute the following calculation.
\begin{enumerate}[(1)]
\item Let $\Delta_i$ be the $(i+1)$st element of $\DDD$.
By~Algorithm~\ref{algo:adjacent}, 
we calculate  Vinberg chambers $\Delta\spar{1}, \dots, \Delta\spar{k}$  in $\VVVS$
such that
$\shortset{\PPP(\Sigma)\cap\Delta\spar{\kappa}}{\kappa=1, \dots, k}$ is the set of induced chambers in $\PPP(\Sigma)$ adjacent  to 
the induced chamber $\PPP(\Sigma)\cap\Delta_i$.
Note that, in~Algorithm~\ref{algo:adjacent},
we also calculate
$\gamma(\Delta\spar{\kappa})\in \OG^+(\Lten)$ for $\kappa=1, \dots, k$. 
\item For $\kappa=1, \dots, k$, we calculate   
$\Sigma\spar{\kappa}:=\sigma(\Delta\spar{\kappa})$ by $\gamma(\Delta\spar{\kappa})$.
\begin{itemize}
\item[(2-1)]
If $\Sigma\spar{\kappa}\notin \sigma(\DDD)$,
then we add $\Delta\spar{\kappa}$ to $\DDD$,  $\gamma(\Delta\spar{\kappa})$ to $\gamma(\DDD)$,
and $\Sigma\spar{\kappa}$ to $\sigma(\DDD)$ at the end of each list.
\item[(2-2)]
If $\Sigma\spar{\kappa}$ appears at the $(m+1)$st position of $\sigma(\DDD)$,
then we have $\Sigma\spar{\kappa}=\sigma(\Delta_m)$, where $\Delta_m\in \DDD$.
We add $g:=\gamma(\Delta\spar{\kappa})\inv \cdot \gamma(\Delta_m)$ to $\GGG$.
Note that, since $\PPP(\Sigma\spar{\kappa})^{\gamma(\Delta\spar{\kappa})}=\PPP(\Sigma)=\PPP(\sigma(\Delta_m))^{\gamma(\Delta_m)}$,
we have $g\in \Stab(\gen{\Sigma}, \Lten)$.
\end{itemize}
\item We increment $i$ to $i+1$.
\end{enumerate}
Since $|\DDD|=|\sigma(\DDD)|$ cannot exceed $|\Sigmas|=1021$,
this algorithm terminates.
When it terminates, we output 
 $\DDD$ as $\DDD_{\Sigma}$, 
 $\sigma(\DDD)$ as $\sigma(\DDD_{\Sigma})$, 
and $\GGG$ as $\GGG_{\Sigma}$.
\end{algorithm}
%
%
%
\begin{proposition}\label{prop:algomain}
We have $\sigma(\DDD_{\Sigma})=\shortset{\Sigma\sprime\in \Sigmas}{\Sigma\sprime\sim\Sigma}$.
Moreover,  the finite set  $\GGG_{\Sigma}$ together with $W(\Sigma, \Lten)$ generates $\Stab(\gen{\Sigma}, \Lten)$.
\end{proposition}
\begin{proof}
Let $G$ be the subgroup of $\OG^+(\Lten)$ generated by the union of $\GGG_{\Sigma}$ and $W(\Sigma, \Lten)$.
Note that $G\subset \Stab(\gen{\Sigma}, \Lten)$,
and hence $\PPP(\Sigma)^g=\PPP(\Sigma)$ holds for all $g\in G$.
First we show that, for any $\Delta\in \VVVS$,
there exists an element $g\in G$ such that $\Delta^g\in \DDD_{\Sigma}$.
An \emph{adjacent sequence} is a sequence 
$D\sbpar{0}, \dots, D\sbpar{N}$
of induced chambers of $\PPP(\Sigma)$ such that $D\sbpar{\nu-1}$ and $D\sbpar{\nu}$ are adjacent
for each $\nu=1, \dots, N$.
The number $N$ is called the \emph{length} of the adjacent sequence
$D\sbpar{0}, \dots, D\sbpar{N}$.
For $\Delta\in \VVVS$, let $d(\Delta)$
be the minimum of the lengths of adjacent sequences from $\PPP(\Sigma)\cap\Delta_0$
to $\PPP(\Sigma)\cap \Delta$.
Since $\PPP(\Sigma)$ is connected and covered by induced chambers,
we have $d(\Delta)<\infty$ for any $\Delta\in \VVV_{\Sigma}$.
Suppose that 
$$
\BBB:=\set{\Delta\in \VVVS}{\textrm{$\Delta^g\notin \DDD_{\Sigma}$ for any $g\in G$}}
$$
is non-empty.
Let $\Delmin\in \BBB$ be an element 
such that $N:=d(\Delmin)\le d(\Delta)$ holds for any $\Delta\in \BBB$.
We put $\Delta\sbpar{0}:=\Delta_0$ and $\Delta\sbpar{N}:=\Delmin$.
Then we have an adjacent sequence
$$
\PPP(\Sigma)\cap\Delta\sbpar{0}, \;\; \PPP(\Sigma)\cap\Delta\sbpar{1}, \;\; \dots, \;\; \PPP(\Sigma)\cap\Delta\sbpar{N-1}, \;\; 
\PPP(\Sigma)\cap\Delta\sbpar{N}
$$
 of  length $N$ from $\PPP(\Sigma)\cap\Delta_0$ to $\PPP(\Sigma)\cap\Delmin$.
The minimality of  $N=d(\Delmin)$ implies that 
there exists an element $g$ of $G$ such that
$\Delta\sbpar{N-1}^g\in \DDD_{\Sigma}$.
We put $\Delta_i:=\Delta\sbpar{N-1}^g$.
Since $\PPP(\Sigma)^g=\PPP(\Sigma)$,
we see that $\PPP(\Sigma) \cap \Delta\sbpar{N}^g$ is an induced chamber adjacent to
$\PPP(\Sigma) \cap \Delta_i$.
Therefore,
when Algorithm~\ref{algo:main} processed $\Delta_i\in \DDD$,
there must be a Vinberg  chamber $\Delta\spar{\kappa}$ 
among $\Delta\spar{1}, \dots, \Delta\spar{k}$ such that
$\PPP(\Sigma) \cap \Delta\spar{\kappa}=\PPP(\Sigma) \cap \Delta\sbpar{N}^g$.
By Lemma~\ref{lem:WSigma},
we have an element $g\sprime\in W(\Sigma, \Lten)\subset G$
such that $\Delta\sbpar{N}^{gg\sprime}=\Delta\spar{\kappa}$.
Since $\Delmin\in \BBB$ implies that $\Delta\spar{\kappa}=\Delmin^{gg\sprime}$ is not in $ \DDD_{\Sigma}$, 
the case (2-2) must have occurred.
Therefore  $\sigma(\Delta\spar{\kappa})=\sigma(\Delmin^{gg\sprime})$ should have  been added to $\sigma(\DDD_{\Sigma})$,
and there exists an element $\Delta_m\in \DDD_{\Sigma}$
such that 
$$
g\spprime :=\gamma(\Delta\spar{\kappa})\inv \cdot \gamma(\Delta_m)\;\;\in\;\;  \GGG_{\Sigma}.
$$
Then $g g\sprime g\spprime\in G$ and $\Delmin^{g g\sprime g\spprime}=\Delta_m\in \DDD_{\Sigma}$,
which is a contradiction. Thus $\BBB=\emptyset$ is proved.
\par
Next we prove $\sigma(\DDD_{\Sigma})=\shortset{\Sigma\sprime\in \Sigmas}{\Sigma\sprime\sim\Sigma}$.
Lemma~\ref{lem:sigmaVVVS} implies  that all $\Sigma\sprime\in \sigma(\DDD_{\Sigma})$ satisfies $\Sigma\sprime\sim\Sigma$.
Suppose that $\Sigma\sprime\in \SSS$ satisfies $\Sigma\sprime\sim\Sigma$,
and let $h\in \OG^+(\Lten)$ be an isometry  such that $\gen{\Sigma}=\gen{\Sigma\sprime}^h$.
We put $\Delta:=\Delta_0^h$.
Since $(\PPP(\Sigma\sprime)\cap \Delta_0)^h=\PPP(\Sigma)\cap\Delta$,
we have  $\Delta\in \VVVS$.
Since $\BBB=\emptyset$,
we have an element $g\in G$ such that $\Delta^g\in \DDD_{\Sigma}$.
Since $\PPP(\Sigma)^g=\PPP(\Sigma)$, we have
$(\PPP(\Sigma\sprime)\cap \Delta_0)^{hg}=\PPP(\Sigma)\cap\Delta^g$,
which implies 
$\Sigma\sprime=\sigma(\Delta^g)$.
Therefore $\Sigma\sprime$ belongs to  $\sigma(\DDD_{\Sigma})$.
\par
We prove $G\supset\Stab(\gen{\Sigma}, \Lten)$.
Let $h$ be an  element of $\Stab(\gen{\Sigma}, \Lten)$.
Then we have $\Delta_0^h\in \VVVS$.
Since $\BBB=\emptyset$,
we have an element $g\in G$ such that $\Delta_0^{hg}\in \DDD_{\Sigma}$.
Since $\PPP(\Sigma)^{hg}=\PPP(\Sigma)$, we have
$(\PPP(\Sigma)\cap \Delta_0)^{hg}=\PPP(\Sigma)\cap\Delta_0^{hg}$,
and therefore $\sigma(\Delta_0^{hg})=\Sigma$.
Since elements of $\sigma(\DDD_{\Sigma})$ are all distinct,
we have $\Delta_0^{hg}=\Delta_0$ and hence $h=g\inv\in G$ holds.
\end{proof}
We execute Algorithm~\ref{algo:main} for all $\Sigma\in \Sigmas$,
and confirm that $\Sigma\sim \Sigma\sprime$ holds
if and only if  their $\ADE$-types $\tau(\Sigma)$ and $\tau(\Sigma\sprime)$ coincide.
Hence we obtain the following:
\begin{theorem}\label{thm:TTT}
The set $\NNN/\OG^+(\Lten)$ is identified with the set 
$\tau(\Sigmas)$ of $\ADE$-types of elements of $\Sigmas$.
In particular, we have $|\,\NNN/\OG^+(\Lten)|=86$.
\qed
\end{theorem}
Thus we have classified, up to the action of $\OG^+(\Lten)$,
all primitive sublattices of $\Lten$ that appear as $\primR_f$ in Theorem~\ref{thm:L10main}.
In order to complete the proof of  Theorem~\ref{thm:L10main},
we now enumerate all $\ADE$-configurations $\Phi_f$ in each $\primR_f$.
\subsection{Stabilizer subgroups}
Note that $W(\Sigma, \Lten)$ is a normal subgroup of the stabilizer subgroup $\Stab(\gen{\Sigma}, \Lten)$.
We put
$$
\Stab(\Sigma, \Lten):=\set{g\in \OG^+(\Lten)}{\Sigma^g=\Sigma},
$$
which is a subgroup of $\Stab(\gen{\Sigma}, \Lten)$ and is mapped isomorphically to
the quotient group $\Stab(\gen{\Sigma}, \Lten)/W(\Sigma, \Lten)$.
Hence the rows of the following commutative diagram are \emph{splitting} exact sequences:
\begin{equation}\label{eq:commsplits}
\begin{array}{ccccccccc}
1 & \to & W(\Sigma, \Lten) & \to & \Stab(\gen{\Sigma}, \Lten) & \maprightsp{\tilde{\kappa}} &\Stab(\Sigma, \Lten) & \to & 1\phantom{,} \\
  & &\mapdownright{\wr} & &\mapdownright{\res} & &\mapdownright{\res} &  &\\
1 & \to & W(\gen{\Sigma}) & \to & \OG(\gen{\Sigma}) & \maprightsp{\kappa} & \Aut(\Sigma) &\to &1,
\end{array}
\end{equation}
where the vertical arrows are the restriction homomorphisms. 
\begin{corollary}
The group $\Stab(\gen{\Sigma}, \Lten)$ is generated by the union of $W(\Sigma, \Lten) $ and $\tilde{\kappa}(\GGG_{\Sigma})$,
and the group  $\Stab(\Sigma, \Lten) $ is generated by $\tilde{\kappa}(\GGG_{\Sigma})$.
\qed
\end{corollary}
We can calculate $\tilde{\kappa}(\GGG_{\Sigma})$ by the same method as Algorithm~\ref{algo:phi}.
We can also calculate a  generating set $\kappa(\res(\GGG_{\Sigma}))=\res(\tilde{\kappa}(\GGG_{\Sigma}))$
of the subgroup
\begin{equation}\label{eq:HSigma}
H_\Sigma:=\Image(\;\; \Stab(\gen{\Sigma}, \Lten)\maprightsp{\res}  \OG(\gen{\Sigma})\maprightsp{\kappa} \Aut(\Sigma)\;\; )
\end{equation}
of $\Aut(\Sigma)$.
%
%
%
%
\subsection{Embeddings of an $\ADE$-configuration}\label{subsec:embLten}
Let $\Phi$ be an $\ADE$-configuration with $|\Phi|<10$.
Let $\Emb(\Phi)$ denote the set of all embeddings of $\Phi$ into $\Lten$.
Then $\Aut(\Phi)$ acts on $\Emb(\Phi)$ from the left,
and $\OG^+(\Lten)$ acts on $\Emb(\Phi)$ from the right.
In this section, we calculate the set $\Aut(\Phi)\backslash \Emb(\Phi)/ \OG^+(\Lten)$,
and prove Theorem~\ref{thm:L10main}.
\par
Let $f\colon \Phi\inj\Lten$ be an embedding.
We denote by $\Phi_{f}$ the image of $f$, 
by $R_f$ the sublattice of $\Lten$ generated by  $\Phi_{f}$,
and by  $\primR_f$ the primitive closure of $R_f$ in $\Lten$.
Then  $\primR_f$  corresponds to  an even overlattice $R_{(f)}$ of $\gen{\Phi}$
via the isometry $\gen{\Phi} \isom R_f$ given by $f$,
and  we have $\primR_f\in \NNN$  by Lemma~\ref{lem:primR}.
Recall that $\tau(\Sigmas)$ is the set of $\ADE$-types of all $\Sigma\in \Sigmas$.
We consider the following condition on an even overlattice $\primR$
of $\gen{\Phi}$.
$$
\textrm{($\sharp$)\;\; $\primR$ is a root lattice whose  $\ADE$-type belongs to $\tau(\Sigmas)$.}
$$
Then Theorem~\ref{thm:TTT} implies the following:
\begin{enumerate}[(a)]
\item 
For any $f\in \Emb(\Phi)$,
the even overlattice $R_{(f)}$ of $\gen{\Phi}$ satisfies  ($\sharp$),
and there exist an element $g\in \OG^+(\Lten)$ and an element $\Sigma\in \SSS$
such that $\primR_{fg}=\gen{\Sigma}$.
\item 
Suppose that an even overlattice $\primR$  of $\gen{\Phi}$ satisfies  ($\sharp$).
Then there exist an element $\Sigma\in \SSS$ and an isometry $\bar{f}\colon \primR\isom \gen{\Sigma}$.
The restriction of $\bar{f}$ to $\Phi\subset \primR$ gives an embedding  $f\in \Emb(\Phi)$.
\end{enumerate}
Therefore we can calculate $\Aut(\Phi)\backslash \Emb(\Phi)/ \OG^+(\Lten)$
by the following method.
\par
\smallskip
(1) 
Let $\OL(\Phi)$ denote the set of even overlattices of $\gen{\Phi}$,
on which $\Aut(\Phi)$ acts from the right.
We calculate the set $\OL(\Phi)/\Aut(\Phi)$ of orbits of this action,
and for each orbit $o\in \OL(\Phi)/\Aut(\Phi)$,
we choose a representative $\primR\in o$
and calculate the stabilizer subgroup $\Stab(\primR, \Phi)$ of $\primR$ in 
the finite group $\Aut(\Phi)$.
\par
\smallskip
(2)
For $\primR\in \OL(\Phi)$,
let $[\primR]\in \OL(\Phi)/\Aut(\Phi)$ denote 
the orbit containing $\primR$.
Similarly,
for $\Sigma\in \Sigmas$,
let $[\Sigma]\in \NNN/\OG^+(\Lten)$ denote 
the orbit containing $\gen{\Sigma}\in \NNN$.
We define a subset $\III_{\Phi}$ of $ \OL(\Phi)/\Aut(\Phi)\times \NNN/\OG^+(\Lten)$ by 
$$
\III_{\Phi}:=\set{([\primR], [\Sigma])}{%
\textrm{$\primR$ satisfies ($\sharp$) and $\tau(\primR)=\tau(\Sigma)$}}.
$$
For each pair $([\primR], [\Sigma])\in \III_{\Phi}$,
we define  a set  $\baremb([\primR], [\Sigma])$ as follows.
Let $\Isom(\primR,\gen{\Sigma})$ denote 
the set of all isomorphisms from $\primR$ to $\gen{\Sigma}$,
on which $\Stab(\primR, \Phi)$ acts from the left and $\Stab(\gen{\Sigma}, \Lten)$ acts from the right.
We put 
$$
\baremb([\primR], [\Sigma]):=\Stab(\primR, \Phi)\backslash \Isom(\primR, \gen{\Sigma})/\Stab(\gen{\Sigma}, \Lten).
$$
Then the set  $\Aut(\Phi)\backslash \Emb(\Phi)/ \OG^+(\Lten)$ is the disjoint union of
$\baremb([\primR], [\Sigma])$,
where  $([\primR], [\Sigma])$ runs through the set $\III_{\Phi}$.
%
%
\subsubsection{Execution of task {\rm (1)}}
The task (1) above can be easily carried out by Proposition~\ref{prop:overlattices}.
We calculate the set of all  totally isotropic subgroups of the discriminant form $\discf{\gen{\Phi}}$ up to the action of $\Aut(\Phi)$.
Executing this calculation for all $157$ $\ADE$-configurations $\Phi$ with $|\Phi|<10$, 
we obtain the following:
\begin{theorem}\label{thm:OL}
Let $\Phi$ be an $\ADE$-configuration with $|\Phi|<10$.
Suppose that $\primR_1$ and $\primR_2$ are even overlattices of $\gen{\Phi}$.
If $\primR_1$ and $\primR_2$ are root lattices
with the same $\ADE$-type,
then there exists an automorphism $g\in \Aut(\Phi)$ such that $\primR_1^g=\primR_2$.
\qed
\end{theorem}
\begin{example}
Suppose that $\Phi=\{r_1, \dots, r_n\}$ is of $\ADE$-type $n A_1$.
Then the discriminant group $\discg{\gen{\Phi}}$ of $\gen{\Phi}$  is an $n$-dimensional $\F_2$-vector space with basis
$ r\dual_i:=-r_i/2 \bmod \gen{\Phi}$ ($i=1, \dots, n$),
and   $\discf{\gen{\Phi}}$ is given by
$$
\discf{\gen{\Phi}}(a_1 r\dual_1+\dots+a_n r\dual_n)=-(a_1^2+\cdots +a_n^2)/2 \bmod 2\Z,
\quad\textrm{where $a_1, \dots, a_n\in \F_2$}.
$$
Hence a subspace $\CCC$ of $\discg{\gen{\Phi}}=\F_2^n$ is  totally isotropic with respect to $\discf{\gen{\Phi}}$
if and only if $\CCC$ is a doubly-even linear code in $\F_2^n$.
(See Ebeling~\cite{EbelingBook} for the terminologies on codes.)
We classify these codes up to the action of $\Aut(\Phi)\cong \SSSS_n$
by a brute-force method.
In the case $n=8$,
there exist exactly $8$ doubly-even linear codes in $\F_2^8$ up to $\SSSS_8$.
The corresponding even overlattices  of $\gen{\Phi}$ are given in Table~\ref{table:8A1},
where $[i_1 \dots i_k]$ denotes the codeword $(r\dual_{i_1}+\cdots +r\dual_{i_k}) \bmod \gen{\Phi}$ in $\F_2^8$.
\begin{table}
$$
\renewcommand{\arraystretch}{1.4}
\begin{array}{cccc}
\dim \CCC & \CCC & \tau(\primR) &\textrm{Condition ($\sharp$)}\\
\hline
 0 & 0 &  8A_1 &\textrm{no}\\
 \hline
 1 & \gen{[5678]} & 4A_1+ D_4 &\textrm{no}\\
 \hline
 1 & \gen{[12\dots 8]} & \textrm{not a root lattice} &\textrm{no} \\
 \hline
 2 & \gen{[3678], \;\;[4578]} & 2 A_1+ D_6 &\textrm{no}\\
 \hline
 2 & \gen{[1678], \;\;[2345]} & 2 D_4 &\textrm{no}\\
   \hline
 3 & \gen{[2678], \;\;[3578], \;\; [4568]}&  A_1+E_7 & \textrm{yes}  \\
  \hline
3 &\gen{[1678], \;\; [2578],\;\; [3478]}&  D_8 &\textrm{yes}  \\
 \hline
 4 & \gen{[1678], \;\; [2578],\;\; [3568], \;\; [4567]}&  E_8 &\textrm{yes}  
\end{array}
%
$$
\vskip 1mm
%
\caption{Even overlattices of the root lattice of type $8A_1$}\label{table:8A1}
\end{table}
\end{example}
%
%
%
\subsubsection{Execution of task {\rm (2)}}
%
Suppose that $([\primR], [\Sigma])$ is an element of  $ \III_{\Phi}$.
We describe a method  to calculate the set $\baremb([\primR], [\Sigma])$.
We can find an $\ADE$-basis of the even overlattice $\primR$ 
by the method described in Section~\ref{subsec:negdefroot},
and hence 
we can write an element $\varphi_0\in \Isom(\primR, \gen{\Sigma})$ explicitly.
Using $\varphi_0$ as a reference point, we identify
 $\Isom(\primR, \gen{\Sigma})$  with $\OG(\gen{\Sigma})$.
 Moreover, since we have a natural injective homomorphism $\Stab(\primR, \Phi)\inj \OG (\primR)$,
we can  regard  $\Stab(\primR, \Phi)$ as a subgroup of $\OG(\gen{\Sigma})$.
\begin{remark}\label{rem:varphi0}
By Algorithm~\ref{algo:phi},
we can choose $\varphi_0\in \Isom(\primR, \gen{\Sigma})$
in such a way that the image by $\varphi_0$ of the connected component
$$
\set{x\in \primR\tensor\R}{\intf{x, r}>0\;\;\textrm{for all}\;\; r \in \Phi}
$$
of $(\primR\tensor\R)\spcirc=(\gen{\Phi}\tensor\R)\spcirc$ contains the connected component
$$
\set{y\in \gen{\Sigma}\tensor\R}{\intf{y, r}>0\;\;\textrm{for all}\;\; r \in \Sigma}
$$
of $(\gen{\Sigma}\tensor\R)\spcirc$. See Section~\ref{subsec:negdefroot}.
\end{remark}
Let $H_{\Phi}$ denote the image of $\Stab(\primR, \Phi)\subset \OG(\gen{\Sigma})$
by the quotient  homomorphism $\OG(\gen{\Sigma})\to \Aut(\Sigma)$
by $W(\gen{\Sigma})$,
which can be calculated by Algorithm~\ref{algo:phi}.
Since $\Stab(\gen{\Sigma}, \Lten)$ contains $W(\Sigma, \Lten)$,
we see that
$$
\baremb([\primR], [\Sigma]):=\Stab(\primR, \Phi) \backslash \OG(\gen{\Sigma})\,/\,\Stab(\gen{\Sigma}, \Lten)
 =H_{\Phi} \backslash \Aut(\Sigma) \,/\, H_{\Sigma}, 
$$
where $H_{\Sigma}$ is defined by~\eqref{eq:HSigma}.
By this computation, we obtain the following:
\begin{theorem}\label{thm:emb}
The set\,  $\baremb([\primR], [\Sigma])$
consists of a single element 
for any $\ADE$-configuration $\Phi$ with $|\Phi|<10$
and any pair $([\primR], [\Sigma])\in \III_{\Phi}$.
\qed
\end{theorem}
Now we can prove Theorem~\ref{thm:L10main}.
\begin{proof}[Proof of Theorem~\ref{thm:L10main}]
The assertion (1)  follows from  Lemma~\ref{lem:primR}.
The assertion (2)  follows from
Theorems~\ref{thm:TTT},~\ref{thm:OL}~and~\ref{thm:emb}.
\end{proof} 
\subsection{The stabilizer subgroup of $\Phi_f$ in $\OG^+(\Lten)$}\label{subsec:stabPhiL}
Suppose that $f\in \Emb(\Phi)$ and $\Sigma\in \Sigmas$
satisfy $\primR_f=\gen{\Sigma}$,
and that $f$ induces an element $\varphi_0\in \Isom(\primR, \gen{\Sigma})$ satisfying 
the condition in Remark~\ref{rem:varphi0}.
We put
$$
\Stab(\Phi_f, \Lten):=\set{g\in \OG^+(\Lten)}{\Phi_f^g=\Phi_f}.
$$ 
It is obvious that $\Stab(\Phi_f, \Lten)\subset \Stab(\gen{\Sigma}, \Lten)$.
The purpose of this section is 
to calculate a finite generating set  $\GGG\sprime_{\Phi}$  of $\Stab(\Phi_f, \Lten)$.
This set $\GGG\sprime_{\Phi}$ will be used in the next section
for the classifications of strongly equivalence classes of $\RDP$-Enriques surfaces.
\par
The following general algorithm is used several times in this section.
\begin{algorithm}\label{algo:generalStab}
Let $G$ be a group generated by $\gamma_1, \dots, \gamma_N \in G$.
Suppose that $G$ acts on a finite set $S$,
and let $s_0$ be an element of $S$.
This algorithm calculates 
a set $\RRR_0$ of elements of $G$ such that $g\mapsto s_0^g$
gives a bijection from $\RRR_0$ 
to the orbit $s_0^G:=\shortset{s_0^g}{g\in G}$ of $s_0$ under the action of $G$.
This algorithm also calculates 
a finite generating set $\GGG_0$ of the stabilizer subgroup
$$
\Stab(s_0, G):=\set{g\in G}{s_0^g=s_0}.
$$
We set  
$\Gamma:=\{\gamma_1, \dots, \gamma_N, \gamma_1\inv, \dots, \gamma_N\inv\}$.
We then put $h_0:=\id\in G$, and
$$
\RRR:=[h_0], 
\quad 
\OOO:=[s_0],
\quad
\GGG:=\{\;\;\},
\quad
i:=0.
$$
During the calculation,
we have the following:
\begin{enumerate}[(i)]
\item $\RRR$ is a list $[h_0, h_1, \dots, h_j]$ of elements of $G$, and  $\OOO$ is the list $[s_0, s_1, \dots, s_j]$ of 
\emph{distinct} elements of $s_0^G$
such that $s_{\mu}=s_0^{h_{\mu}}$ holds for $\mu=0, \dots, j$, and
\item $\GGG$ is a set of elements of $\Stab(s_0, G)$.
\end{enumerate}
While $i+1\le j+1=|\RRR|=|\OOO|$, we execute the following calculation.
\begin{enumerate}[(1)]
\item Let $h_i$ be the $(i+1)$st element of $\RRR$,
and let $s_i=s_0^{h_i}$ be the $(i+1)$st element of $\OOO$.
For each $\gamma\in \Gamma$,
we execute the following:
\begin{itemize}
\item[(1-1)]
If $s_i^{\gamma}=s_0^{h_i\gamma}\notin\OOO$, then we add $h_i {\gamma}$ to $\RRR$ and $s_i^{\gamma}$ to $\OOO$
at the end of each  list, whereas
\item[(1-2)]
if  $s_i^{\gamma}=s_0^{h_i\gamma}$ is equal to the $(m+1)$st element $s_m=s_0^{h_m}$ of $\OOO$,
then we add $h_i\gamma h_m\inv \in \Stab(s_0, G)$ to $\GGG$.
\end{itemize}
\item We increment $i$ to $i+1$.
\end{enumerate}
Since $|\OOO|=|\RRR|$ cannot exceed $|S|$,
this algorithm terminates.
When it terminates, 
it outputs  $\RRR$ as $\RRR_0$, $\OOO$ as $\OOO_0$, and $\GGG$ as  $\GGG_0$.
\end{algorithm}
%
%
\begin{proposition}\label{prop:algostabPhiL}
The set $\OOO_0$ is equal to $s_0^G$, and $\GGG_0$ generates $\Stab(s_0, G)$.
\end{proposition}
\begin{proof}
The proof is similar to 
 the proof of Proposition~\ref{prop:algomain}.
The details are left to the reader.
\end{proof}
\begin{remark}
We have  calculated a finite generating set  $\tilde{\kappa}(\GGG_{\Sigma})\cup W(\Sigma, \Lten)$
of $\Stab(\gen{\Sigma}, \Lten)$,
and hence we can obtain a finite generating set of $\Stab(\Phi_f, \Lten)$
by applying Algorithm~\ref{algo:generalStab} to the action of $\Stab(\gen{\Sigma}, \Lten)$ on the  set 
of $\ADE$-configurations of roots in $\gen{\Sigma}$. 
This method, however, takes too much time and memory for many $f\in \Emb(\Phi)$.
Therefore we use the following method.
\end{remark}
We put
$$
V:=\gen{\Sigma}\tensor\R=\gen{\Phi_f}\tensor \R.
$$
Then $\Stab(\gen{\Sigma}, \Lten)$ acts on $V$ via the restriction homomorphism $\res$ (see~\eqref{eq:commsplits}).
We denote the action of $g\in \Stab(\gen{\Sigma}, \Lten)$ on $V$ simply
by $v\mapsto v^g$
without writing $\res$.
For an $\ADE$-configuration $\Psi$ of roots of $\gen{\Sigma}$, we put
$$
\Gamma(\Psi):=\set{y\in V}{\intf{y, r}\ge 0\;\;\textrm{for all}\;\; r\in \Psi}.
$$
By the assumption on $f$ (see Remark~\ref{rem:varphi0}), we have
$$
\Gamma(\Sigma)\subset \Gamma(\Phi_f).
$$
We calculate the finite set
$$
\WWW(\Phi_f, \Sigma):=\set{w\in W(\Sigma, \Lten)}{\Gamma(\Sigma)^w\subset \Gamma(\Phi_f)}.
$$
Recall that $\tilde{\kappa}(\GGG_{\Sigma})$ is a generating set of 
the subgroup $\Stab(\Sigma, \Lten)$ of  $\Stab(\gen{\Sigma}, \Lten)$.
Therefore, by Algorithm~\ref{algo:generalStab},  we can calculate a finite subset
$$
\RRR_{\Phi_f, \Sigma}=\{h_1, \dots, h_N\}
$$
 of $\Stab(\Sigma, \Lten)$ such that
$g\mapsto \Phi_f^g$ gives a bijection from $\RRR_{\Phi_f,  \Sigma}$
to the orbit 
$$
\OOO(\Phi_f):=\set{\Phi_f^g}{g\in \Stab(\Sigma, \Lten)}.
$$ 
\par
We set
$\GGG\sprime:=\{\;\;\}$.
When the calculation below terminates,
this set $\GGG\sprime$ is the desired generating set $\GGG\sprime_{\Phi}$ of $\Stab(\Phi_f, \Lten)$.
Recall that  the upper row of~\eqref{eq:commsplits} is a splitting exact sequence.
Hence each element $g$ of $\Stab(\gen{\Sigma}, \Lten)$ is 
uniquely written as $\tilde{\kappa}(g) w$,
where $w\in W(\Sigma, \Lten)$.
We consider the coset decomposition
$$
\Stab(\gen{\Sigma}, \Lten)=\bigsqcup_{w\in W(\Sigma, \Lten)}\; \Stab(\Sigma, \Lten)\cdot w, 
$$
and for each $w\in W(\Sigma, \Lten)$, we consider the set
$$
\Xi(w):=( \Stab(\Sigma, \Lten)\cdot w)\cap \Stab(\Phi_f, \Lten).
$$
If $w\notin \WWW(\Phi_f, \Sigma)$, then we obviously have $\Xi(w)=\emptyset$.
Therefore we assume that $w\in  \WWW(\Phi_f, \Sigma)$.
We then calculate the image 
$\OOO(\Phi_f)^w$ of the orbit $\OOO(\Phi_f)$ by $w$.
If $\OOO(\Phi_f)^w$ 
does not contain $\Phi_f$, then  we  have $\Xi(w)=\emptyset$.
Therefore we assume that  there exists an element $h_i$ of $\RRR_{\Phi_f,  \Sigma}$ such that $\Phi_f^{h_i w}=\Phi_f$.
We add $h_i w$ in $\GGG\sprime$.
Every element of $\Xi(w)$ is written uniquely as $h_i g w$, where $g\in \Stab(\Sigma, \Lten)$.
Since $h_i g w = h_i w (w\inv g w)$, 
an element $g\in \Stab(\Sigma, \Lten)$ satisfies $h_i g w\in \Xi(w)$ if and only if
$(\Phi_f^{w\inv})^g =\Phi_f^{w\inv}$.
Using the finite generating set $\tilde{\kappa}(\GGG_{\Sigma})$ of $\Stab(\Sigma, \Lten)$
and Algorithm~\ref{algo:generalStab},
we calculate a finite generating set $\GGG\spprime (w)$ of
$$
\Stab(\Phi_f^{w\inv}, \Sigma, \Lten):=\set{g\in \Stab(\Sigma, \Lten)}{(\Phi_f^{w\inv})^g =\Phi_f^{w\inv}}.
$$
We then append $w\inv \GGG\spprime(w) w$ to $\GGG\sprime$.
%
Thus a finite generating set  $\GGG\sprime_{\Phi}$ of $\Stab(\Phi_f, \Lten)$ is computed.
\par
\medskip
The group $\Stab(\Phi_f, \Lten)$ is,  in general,  infinite.
However we have examples as follows.
\begin{example}
We consider the case where  $\tau(\Phi)=8 A_1$ and $\tau(\Sigma)=E_8$
(no.~88 of Table~\ref{table:main1}).
In this case,
the method above 
still takes too much computation time,
but we can calculate $\Stab(\Phi_f, \Lten)$ as follows.
Since $\primR_f= \gen{\Sigma}$ is unimodular,
$\Lten$ is the orthogonal direct-sum of $\gen{\Sigma}$
and a hyperbolic plane $U$.
Hence $\Stab(\gen{\Sigma}, \Lten)$ is contained in the subgroup 
$\OG(\gen{\Sigma})\times \OG^+(U)$ of $\OG^+(\Lten)$.
Note that $\OG^+(U)$ is of order $2$.
Let $g_U$ be the non-trivial element of $\OG^+(U)$.
Then the kernel of the natural homomorphism
$$
\Stab(\Phi_f, \Lten) \to \Aut(\Phi)
$$
is generated by $(\id, g_U)\in \OG(\gen{\Sigma})\times \OG^+(U)$,
and the image of this homomorphism is
$\shortset{\sigma\in \Aut(\Phi)\cong \SSSS_8}{\CCC^{\sigma}=\CCC}$, 
where $\CCC$ is the  doubly-even linear code in $\discg{\gen{\Phi}}\cong \F_2^8$
corresponding to the even overlattice $\primR\cong\gen{\Sigma}$,
that is, $\CCC$ is an extended Hamming code~\cite[Chapter 1]{EbelingBook}.
Therefore we have $|\Stab(\Phi_f, \Lten)|=2688$,
and we can easily obtain a generating set of $\Stab(\Phi_f, \Lten)$.
\par 
The case where  $\tau(\Phi)=9 A_1$ and $\tau(\Sigma)=A_1+E_8$
(no.~146 of Table~\ref{table:main1}) is also treated in the similar method.
We have $|\Stab(\Phi_f, \Lten)|=1344$ in this case.
\end{example}
\section{Geometric realizability}\label{sec:geom}
Let $\Phi$ be an $\ADE$-configuration with $|\Phi|<10$.
An embedding $f\colon \Phi\inj\Lten$ is said to be \emph{geometrically realized
by an $\RDP$-Enriques surface $(Y, \rho)$} if there exists an isometry
$\SY\isom \Lten$ that maps $\PPP_Y$ to $\Pten$ and $\Phi_{\rho}$ to  $\Phi_{f}$  bijectively. 
If  an embedding $f\sprime\colon \Phi\sprime\inj\Lten$ is equivalent (in the sense of Definition~\ref{def:equivL10})
to a geometrically  realizable embedding $f\colon \Phi\inj\Lten$, 
then $f\sprime$ is also geometrically realizable.
The purpose of this section is to introduce a lattice $\primM_f$
corresponding to the lattice $\primM_{\rho}$ associated with  an $\RDP$-Enriques surface $(Y, \rho)$, and 
give a criterion 
for the geometric realizability. 
By means of these tools, we classify 
the strong equivalence classes of $\RDP$-Enriques surfaces,  and prove Theorem~\ref{thm:strongmain}. 
\par
For a lattice $L$,
let $L(2)$ denote the lattice obtained from $L$ by multiplying 
the intersection form  
by $2$. 
The orthogonal direct-sum of lattices $L$ and $L\sprime$
is denoted by $L\oplus L\sprime$.
\subsection{Geometry of an Enriques involution}\label{subsec:geomenr}
An involution $\enrinvol\colon X\to X$
of a $K3$ surface $X$ is said to be an  \emph{Enriques involution}
if $\enrinvol$ has no  fixed points,
or equivalently, the quotient $X/\gen{\enrinvol}$ is an Enriques surface.
The  following  follows from the classification 
of $2$-elementary $K3$ surfaces 
due to  Nikulin~\cite{Nikulin1980}.
\begin{proposition}[Nikulin]  \label{prop:enrinvolNikulin}
Let $X$ be a $K3$ surface,
and let $\enrinvol\colon X\to X$ be an involution,
which acts on $H^2(X,\Z)$ from the right by the pull-back.
Suppose that $\enrinvol$ acts on $H^0(X, \Omega_X^2)$ as the multiplication by $-1$.
Let $\SX^+$ denote  the primitive sublattice $\shortset{v\in\SX}{v^\enrinvol=v}$
of $\SX$.
Then $\enrinvol$ is  an Enriques involution if and only if
$\SX^+$ is isomorphic to $\Lten (2)$.
\qed
\end{proposition}
Let $\enrinvol\colon X\to X$ be an Enriques involution,
and let $\pi\colon X\to Y$ denote  the universal covering of the Enriques surface $Y:=X/\gen{\enrinvol}$.
Then $\SY$ is isomorphic to $\Lten$,
and $\pi^*\colon \SY\to \SX$
induces an isometry $\SY(2)\cong \pi^*\SY=\SX^+$.
We also have $\pi^{*-1}(\PPP_X)=\PPP_Y$.
By an isometry $\SY\cong \Lten$ that maps $\PPP_Y$ to $\PPP_{10}$,
the notion of Vinberg chambers  in $\PPP_Y$ makes sense.
We also use the notation~\eqref{eq:Sperp1} and~\eqref{eq:Sperp2} for $\SY$.
We put
\begin{eqnarray*}
\Nef(X) &:=& \set{x\in \PPP_X}{\intf{x, C\sprime} \ge 0 \;\;\textrm{for all curves}\;\; C\sprime\subset X}, \\
\Nef(Y) &:=& \set{x\in \PPP_Y}{\,\intf{x, C} \ge 0 \;\;\,\textrm{for all curves}\;\; C\;\subset Y}. 
\end{eqnarray*}
It is well-known that  $\Nef(X)$ is a standard fundamental domain of the action of $W(\SX)$ on $\PPP_X$.
It is obvious that $\pi^{*-1}(\Nef(X))=\Nef(Y)$.
Since $\Nef(Y)$ is bounded by the hyperplanes $([C])\sperp$ of $\PPP_Y$
perpendicular to the classes of smooth rational curves $C$ on $Y$,
and a smooth rational curve on $Y$ has self-intersection number $-2$, 
the cone $\Nef(Y)$ is tessellated by Vinberg chambers.
\par
In the following,
we fix an ample class $a\in \SY$
such that $\intf{r, a}\ne 0$ for any root $r$ of $\SY$.
For example, we choose an interior point of a Vinberg chamber contained in $\Nef(Y)$.
In particular,
if $R$ is a negative-definite root sublattice of $\SY$,
then $\intf{-, a}$ is an element of $\Hom(R, \R)\spcirc$ in the notation of Section~\ref{subsec:negdefroot}.
Since $\pi^*a$ is  ample on $X$,
we have  $\intf{r\sprime, \pi^*a}\ne 0$ for any root $r\sprime$ of $\SX$.
We put
$$
N_{Y}:=\set{v\in \SX}{\intf{v, y}=0\;\; \textrm{for all}\;\; y\in \pi^*\SY}.
$$
If $N_{Y}$ contained a root, then there would be an effective divisor of $X$ contracted by $\pi$,
which is absurd.
Hence we have the following:
\begin{lemma}\label{lem:Npi}
The negative-definite even lattice $N_{Y}$ contains no roots.
\qed
\end{lemma}
Note that $\enrinvol$ acts on $N_{Y}$ as the multiplication by $-1$.
We put
$$
\TTT_{Y}:=\set{t\in N_{Y}}{\intf{t,t}=-4}.
$$
If $r\in \Roots(\SY)$ and $t\in \TTT_{Y}$, 
then  $(\pi^*r+t)/2\in \SX\tensor\Q$ is of square-norm $-2$.
A root $r$ of $\SY$ is said to be \emph{liftable}
if  $r\sprime:=(\pi^*r+t)/2$ is contained in $ \SX$ for some $t\in \TTT_{Y}$,
and when this is the case,
the root $r\sprime$ of $\SX$  is called a \emph{lift} of $r$.
\begin{lemma}\label{lem:twolifts}
For a liftable root $r\in \SY$, there exist exactly two lifts $r\sprime$ and $r\spprime$ of $r$,
and they satisfy $\pi^*r=r\sprime+r\spprime$, 
$ r\spprime=r\sp{\prime\enrinvol}$,  and $\intf{r\sprime, r\spprime}=0$.
\end{lemma}
\begin{proof}
If $r\sprime=(\pi^*r+t)/2$ with $t\in \TTT_{Y}$ is a lift of $r$, then so is $r\spprime:=(\pi^*r-t)/2$.
Suppose that $t\sprime\in \TTT_{Y}$ 
gives a lift $(\pi^*r+t\sprime)/2$  of $r$.
Then we have $(t-t\sprime)/2\in N_{Y}$.
By Lemma~\ref{lem:Npi}, 
we see that $\intf{t, t\sprime}=-4$ or $\intf{t, t\sprime}\ge 4$.
By the Cauchy-Schwarz inequality for $N_{Y}$, we have $|\intf{t, t\sprime}|\le 4$.
Therefore we have $t\sprime=\pm t$.
\end{proof}
An effective  divisor $D$ of $Y$ is said to be \emph{splitting}
if there exists an effective divisor $D\sprime$  of $X$ such that 
$\pi^*(D)=D\sprime+\enrinvol(D\sprime)$ and $\intf{D\sprime, \enrinvol(D\sprime)}=0$.
Note that an effective divisor of $Y$ is splitting if each connected component of its support is simply connected.
In particular, 
a smooth rational curve on $Y$ is  splitting.
\begin{lemma}\label{lem:liftablesplitting}
Let $r$ be a root of $\SY$ such that $\intf{r, a}>0$.
Then $r$ is the class of a  splitting
effective divisor of $Y$
if and only if $r$ is liftable.
\end{lemma}
\begin{proof}
Suppose that $r=[D]$,  where $D$ is a  splitting
effective divisor, and suppose that $\pi^*(D)=D\sprime+\enrinvol(D\sprime)$ with 
$\intf{D\sprime, \enrinvol(D\sprime)}=0$.
Then $t:=[D\sprime]-[\enrinvol(D\sprime)]$ belongs to $ \TTT_{Y}$ and $[D\sprime]=(\pi^*r+t)/2\in \SX$.
Therefore $r$ is liftable.
Conversely, 
suppose that $r$ has lifts $r\sprime$ and $r\spprime=r^{\prime \enrinvol}$.
Since $r\sprime$ is a root of $\SX$ and $\intf{r\sprime, \pi^*a}>0$,
there exists an effective divisor $D\sprime$ of $X$ such that $r\sprime=[D\sprime]$.
Then we have $\pi^* r =r\sprime+r\spprime$, $r\spprime=[\enrinvol(D\sprime)]$ and $\intf{D\sprime, \enrinvol(D\sprime)}=0$
by Lemma~\ref{lem:twolifts}.
Let $D$ be the effective divisor of $Y$ such that $\pi^*(D)=D\sprime+\enrinvol(D\sprime)$.
Then $D$ is splitting and we have $r=[D]$.
\end{proof}
%
Let $H$ be a nef divisor of $Y$ such that $\intf{H, H}>0$,
and let $h\in \Nef(Y)\cap \SY$ be the class of $H$.
If $k$ is a sufficiently large and divisible integer,
then the complete linear system $|kH|$ is base-point free
and the Stein factorization of the morphism
$Y\to \P^N$ induced by $|kH|$
gives rise to  an $\RDP$-Enriques surface
$$
\rho_h\colon Y\to \Ybar.
$$
We  calculate the set $\Phi_h:=\Phi_{\rho_h}$ of the classes of smooth rational curves contracted by $\rho_h$.
Let $[h]\sperp$ denote the orthogonal complement in $\SY$ of the sublattice $[h]:=\Z h$ generated by $h$.
Then $[h]\sperp$ is negative-definite.
We put
\begin{equation}\label{eq:PihLh}
\Pi_h^+:=\set{r\in \Roots([h]\sperp)}{\intf{r, a}>0},
\;\,
L_h^+:=\set{r\in \Pi_h^+}{\textrm{$r$ is liftable}},
\end{equation}
and consider the root sublattice $\gen{L_h^+}$ of $[h]\sperp$ generated by $L_h^+$
(possibly of rank $0$).
\begin{proposition}\label{prop:Phih}
The $\ADE$-configuration  $\Phi_h$ of  the $\RDP$-Enriques surface $(Y, \rho_h)$
is equal to the $\ADE$-basis of $\gen{L_h^+}$
associated with  the linear form $\intf{-, a}$ 
by the correspondence  described in Section~\ref{subsec:negdefroot}.
\end{proposition}
\begin{proof} 
By Lemma~\ref{lem:liftablesplitting},
every element $r$ of $L_h^+$ is the class 
of a splitting effective divisor $D_r$  contracted by $\rho_h$.
Since the class of any smooth rational curve contracted by $\rho_h$
is in $L_h^+$,
the divisor $D_r$ is irreducible
if and only if $r\in L_h^+$ is indecomposable in  $L_h^+$
in the sense of Definition~\ref{def:indecomp}.
Therefore $\Phi_{h}$ is equal to 
 the set $\Phi_h\sprime$ of indecomposable elements of $L_h^+$.
In particular, the set  $\Phi_h\sprime$ is an $\ADE$-configuration of roots of $\SY$,
and the vectors of $\Phi_h\sprime$ are linearly independent.
Since every element of $L_h^+$ is a linear combination of the  indecomposable elements,
we have $\gen{\Phi_h\sprime}= \gen{L_h^+}$.
Since the $\ADE$-basis of the root lattice $ \gen{L_h^+}$
associated with the linear form $\intf{-, a}\in \Hom(\gen{L_h^+}, \R)\spcirc$ is unique,
we obtain the proof.
\end{proof}
\subsection{Lattices associated with an $\RDP$-Enriques surface}\label{subsec:latticesassociatedwith}
Let $(Y, \rho)$ be an $\RDP$-Enriques surface,
and $a\in \SY$ an ample class
such that $\intf{a, r}\ne 0$ holds for any $r\in \Roots(\SY)$.
Let $C_1, \dots, C_n$ be the smooth rational curves on $Y$  contracted by $\rho\colon Y\to \Ybar$,
so that $\Phi_{\rho}=\{[C_1], \dots, [C_n]\}$.
Since 
the Dynkin diagram of $\Phi_{\rho}$ is 
a disjoint union of trees,
we have the following:
\begin{lemma}\label{lem:simplyconnected}
Let $D$ be an effective divisor of $Y$ contracted by $\rho$.
Then each connected component of the support of $D$ is simply connected.
In particular, $D$ is splitting.
\qed
\end{lemma}
Let $\pi\colon X\to Y$ be the universal covering, 
and  let $C_i\sprime$ and $C_i\spprime$ be the two connected components of $\pi\inv (C_i)$.
Lemma~\ref{lem:simplyconnected} implies that, 
interchanging $C_i\sprime$ and $C_i\spprime$ if necessary,
we can assume that
\begin{equation}\label{eq:CiCj}
\intf{C_i, C_j}=\intf{C_i\sprime, C_j\sprime}=\intf{C_i\spprime, C_j\spprime},
\quad{\rm and}\quad  \intf{C_i\sprime, C_j\spprime}=0
\end{equation}
hold for all $i, j$.
We put
$$
\phi ([C_i]):=[C_i\sprime]-[C_i\spprime]\;\;\in\;\;N_{Y}.
$$
Note that  $[C_i\sprime]$ and $[C_i\spprime]$ are the lifts of $[C_i]$ associated with $\phi([C_i])\in \TTT_{Y}$.
\begin{remark}\label{rem:phichoice}
Let $c$ be the number of connected components 
of the Dynkin diagram of $\Phi_{\rho}$.
Then there exist exactly $2^c$ possibilities for the choice of
the map  $\phi$.
\end{remark}
By~\eqref{eq:CiCj},  we have
\begin{equation}\label{eq:phiCiCj}
\intf{\phi ([C_i]),\phi ([C_j])}=2 \intf{[C_i], [C_j]}, 
\end{equation}
and hence $\phi$ defines an embedding 
\begin{equation}\label{eq:phidef}
\phi\colon \gen{\Phi_{\rho}}(2)\inj N_{Y}.
\end{equation}
Recall that $R_{\rho}$ is the sublattice of $\SY$ generated by $\Phi_{\rho}$,
and that $\primR_{\rho}$ is the primitive closure of $R_{\rho}$ in $\SY$.
 Recall also that 
  $\Phi^{\sim}_{\rho}$ is the subset $\{[C_1\sprime], [C_1\spprime], \dots, [C_n\sprime], [C_n\spprime]\}$ of $\Roots(\SX)$,
  that 
 $M_{\rho}$ is the sublattice of $\SX$ generated by  $ \pi^*\SY$ and 
  $\Phi^{\sim}_{\rho}$, 
 and that  $\primM_{\rho}$ is the primitive closure 
 of  $M_{\rho}$ in $\SX$. 
Then $M_\rho$ is an even overlattice of the orthogonal direct-sum
$$
B_{\rho}:= \pi^*\SY\oplus \Image \phi.
$$
 Thus we obtain  a sequence of inclusions
\begin{equation}\label{eq:rhoseq}
\pi^* \Phi_{\rho}\; \subset\;  
\pi^* R_\rho\; \subset\;  \pi^* \primR_\rho\; \subset\;   \pi^*\SY\;  \subset\; \
B_{\rho} \;  \subset\;   M_\rho\;  \subset\;  \primM_{\rho}.
\end{equation}
%
%
%
\begin{lemma}\label{lem:Phirhotilde}
We have
$\Phi^{\sim}_{\rho}=\set{r\sprime\in \Roots(\primM_{\rho})}{r\sprime+r^{\prime \enrinvol}\in \pi^* \Phi_{\rho}}$.
\end{lemma}
\begin{proof}
The inclusion $\subset$ is obvious.
Suppose that $r\sprime\in \Roots(\primM_{\rho})$
satisfies $r\sprime+r^{\prime \enrinvol}=\pi^*r$ for some $r\in  \Phi_{\rho}$.
Since $\intf{r, a}>0$, we have $\intf{r\sprime, \pi^* a}>0$,
and hence we have an effective divisor $D\sprime$ of $X$
such that $r\sprime=[D\sprime]$.
Let $D$ be the effective divisor of $Y$
such that $\pi^* D=D\sprime +\enrinvol(D\sprime)$.
Then $r=[D]\in \Phi_{\rho}$ and hence $D=C_i$ for some $i$.
Therefore $r\sprime \in \Phi^{\sim}_{\rho}$.
\end{proof}
We can now prove Lemma~\ref{lem:QYrho} stated in the introduction. 
\begin{proof}[Proof of Lemma~\ref{lem:QYrho}]  
Since $\pi^*S_Y$ (resp.~$\pi^{\prime*}S_{Y\sprime}$)
is the invariant sublattice of the Enriques involution on $X$ (resp.~on $X\sprime$),
the strong equivalence isometry $\mu$ is compatible with the action of 
the Enriques involutions on $\primM_{\rho}$ and $\primM_{\rho\sprime}$.
Since $\mu_Y$ maps $\Phi_{\rho}$ to $\Phi_{\rho\sprime}$
bijectively, 
Lemma~\ref{lem:Phirhotilde} implies that  $\mu$ maps $\Phi^{\sim}_{\rho}$ to $\Phi^{\sim}_{\rho\sprime}$
bijectively.
\end{proof}
\begin{proposition}\label{prop:liftrho}
The set $\shortset{r\in \Roots(\primR_{\rho})}{\text{\rm $r$ is liftable}}$ is equal to
$\Roots(R_{\rho})$.
\end{proposition}
\begin{proof}
Let $h_{\rho}\in \SY$ be the class of the pull-back by $\rho\colon Y\to \Ybar$
of an ample divisor  of $\Ybar$.
Since $\Phi_{\rho}$ is orthogonal to $h_{\rho}$,
the sublattices $R_{\rho}$ and $\primR_{\rho}$ of $\SY$ are orthogonal to $h_{\rho}$.
Suppose that $r\in \Roots(\primR_{\rho})$ is liftable.
Replacing $r$ by $-r$ if necessary, 
we can assume that  $\intf{r, a}>0$.
By Lemma~\ref{lem:liftablesplitting}, 
the root $r$ is the class of a splitting effective divisor $D$.
Since $\intf{r, h_{\rho}}=0$, the divisor $D$ is 
contracted by $\rho$.
Therefore  $D$ is a linear combination of $C_1, \dots, C_n$.
Thus we have
$r\in \Roots(R_{\rho})$.
Suppose that  $r\in \Roots(R_{\rho})$.
Replacing $r$ by $-r$ if necessary, 
we can assume that  $\intf{r, a}>0$,
and hence there exists an effective divisor $D=\sum a_i C_i$
such that $r=[D]$.
Then $D$ is splitting by Lemma~\ref{lem:simplyconnected}, 
and hence   $r$ is liftable by Lemma~\ref{lem:liftablesplitting}.
\end{proof}
\subsection{A criterion for geometric realizability}
For an embedding $f\colon \Phi\inj\Lten$
of an $\ADE$-configuration $\Phi=\{r_1, \dots, r_n\}$ into $\Lten$,
we will construct a sequence in~\eqref{eq:fseq}, 
which  is a lattice-theoretic counterpart of~\eqref{eq:rhoseq}.
For $r_i\in \Phi$,
we denote by $r_i^+\in \Phi_{f}$ the image of $r_i$ by $f$, so that 
$\Phi_f=\{r_1^+, \dots, r_n^+\}$.
Let $\Phi^{-}=\{r_1^-, \dots, r_n^-\}$ be a copy of $\Phi$
with a fixed isomorphism $\Phi\cong \Phi^{-}$ denoted by $r_i\mapsto r_i^-$.
We put
$$
B_{\Phi}:=\Lten (2) \oplus \gen{\Phi^{-}}(2).
$$
We denote by
$$
\varpi^*\colon \Lten\isom \Lten(2),
\quad
\varphi\colon  \gen{\Phi^{-}}\isom  \gen{\Phi^{-}}(2),
$$
the identity maps on the underlying $\Z$-modules.
Hence we have 
\begin{equation}\label{eq:connectrirj}
\intf{\varpi^*r_i^+, \varpi^* r_j^+}=
\intf{\varphi(r_i^-), \varphi( r_j^-)}=2\intf{r_i, r_j}.
\end{equation}
We define vectors $r_i\sprime$ and $r_i\spprime$ of $B_{\Phi}\dual\subset B_{\Phi}\tensor\Q$ by 
$$
r_i\sprime:=(\varpi^*r_i^+ + \varphi(r_i^-))/2,
\quad 
r_i\spprime:=(\varpi^* r_i^+ -  \varphi(r_i^-))/2.
$$
Let $M_f$ denote the submodule of $B_{\Phi}\dual$
generated by $B_{\Phi}$ and $r_1\sprime, \dots, r_n\sprime$.
Then $M_f$ is an even overlattice of $B_{\Phi}$.
Note that 
$r_i\spprime$
also belongs to $M_f$.
Let $\primM_{f}$ be an even overlattice of $M_f$.
Recall that $R_f$ is the sublattice of $\Lten$ generated by $\Phi_f$, 
and that $\primR_f$ 
is the primitive closure of $R_f$ in $\Lten$.
Then we have a sequence of inclusions 
\begin{equation}\label{eq:fseq}
\varpi^*\Phi_{f}\;  \subset\;  
\varpi^*R_f\; \subset\;  \varpi^*\primR_f\; \subset\;  \varpi^* \Lten\;  \subset\;  B_{\Phi}\;  \subset\;  M_f\;  \subset\;  \primM_{f}.
\end{equation}
Let $N(\primM_{f})$ denote the orthogonal complement of $ \varpi^* \Lten$ in $\primM_{f}$.
We put
$$
\TTT(\primM_{f}):=\set{t\in N(\primM_{f})}{\intf{t, t}=-4}.
$$
A root $r$ of $\primR_f$ is said to be \emph{$\primM_{f}$-liftable} if there exists an element $t\in \TTT(\primM_{f})$
such that $r\sprime:=(\varpi^*r+t)/2\in \primM_{f}\tensor\Q$ is contained in  $\primM_{f}$,
and when this is the case,
we say that $r\sprime$ is an \emph{$\primM_{f}$-lift} of $r$.
\begin{lemma}\label{lem:rootsRfliftable}
Every root of $R_f$ is $\primM_f$-liftable 
for any even overlattice $\primM_f$ of $M_f$.
\end{lemma}
\begin{proof}
Let $r^+$ be a root of $R_f$.
We can write $r^+$ as $\sum a_i r_i^+$ with $a_i\in \Z$.
Then, putting $r^-:=\sum a_i r_i^-\in \gen{\Phi\sp{-}}$, we have $\varphi (r^-)\in \TTT(M_f)\subset \TTT(\primM_f)$,
and hence $r^+$ has an $\primM_f$-lift $(\varpi^*r^+ +\varphi(r^-))/2=\sum a_i r_i\sprime$.
\end{proof}
\begin{definition}\label{def:conds}
We define the following conditions $\condone$-$\condfour$
on  $\primM_{f}$.
Let $\LK$ denote the $K3$ lattice,
that is, $\LK$ is an even unimodular lattice of signature $(3, 19)$,
which is unique up to isomorphism. 
\par
\medskip
$\condone$
The lattice  $\primM_{f}$ can be embedded primitively into $\LK$.
\par
$\condtwo$
The sublattice $\varpi^*\Lten$ is primitive in $\primM_{f}$.
\par
$\condthree$
The negative-definite even lattice $N(\primM_{f})$ contains no roots.
\par
$\condfour$
The set $\shortset{r\in \Roots(\primR_f)}{\textrm{$r$ is $\primM_{f}$-liftable}}$ is equal to $\Roots(R_f)$.
\end{definition}
%
%
\begin{definition}\label{def:strongreal}
We say that an even overlattice $\primM_f$ of $M_f$ is
\emph{strongly realized} by an $\RDP$-Enriques surface $(Y, \rho)$
if there exists an isometry 
$$
m\colon \primM_{f}\isom \primM_{\rho}
$$
with the following properties;
the isometry $m$ maps $\varpi^*\Lten$ to $\pi^{*} \SY$
isomorphically, 
and the isometry $m_Y\colon \Lten \isom \SY$ induced by $m$ 
maps $\PPP_{10}$ to $\PPP_{Y}$
and $\Phi_{f}$ to $\Phi_{\rho}$ bijectively.
An isometry $m\colon \primM_{f}\isom \primM_{\rho}$  satisfying 
these conditions is called a \emph{strong-realization isometry}.
\end{definition}
\begin{remark} 
If an even overlattice $\primM_f$ of $M_f$ is
strongly realized by an $\RDP$-Enriques surface $(Y, \rho)$,
then $f\colon \Phi\inj \Lten$ is geometrically realized by $(Y, \rho)$.
\end{remark}
\begin{theorem}\label{thm:geom1}
If an embedding $f\colon \Phi\inj\Lten$
is geometrically realized by 
an $\RDP$-Enriques surface $(Y, \rho)$,
then there exist an even overlattice $\primM_{f}$ of $M_f$ 
strongly realized by  $(Y, \rho)$.
\end{theorem} 
\begin{proof} 
By the assumption, 
we have an isometry
$m_Y\colon \Lten\isom \SY$
that maps $\Pten$ to $\PPP_Y$ and $\Phi_{f}$ to $\Phi_{\rho}$ bijectively.
We use the notation about $(Y, \rho)$ fixed in Section~\ref{subsec:latticesassociatedwith}, and 
 index the elements $r_1, \dots, r_n$ of $\Phi$ in such a way that
$m_Y(r_i^+)=[C_i]$ holds
for $i=1, \dots, n$.
Then $m_Y$ induces  an isometry 
$$
m^+ \; \colon\;   \varpi^*\Lten \isom  \pi^*\SY
$$
that maps $\varpi^*r_i^+$ to $[\pi^*(C_i)]$.
By~\eqref{eq:phiCiCj}, 
we also have an isometry 
$$
m^- \; \colon\;  \gen{\Phi^{-}}(2) \isom \Image (\phi\colon \gen{\Phi_{\rho}}(2)\inj N_Y)
$$
that maps $ \varphi(r_i^-)$ to $\phi ([C_i])=[C_i\sprime]-[C_i\spprime]$.
Thus we have an isometry
$$
m:=m^+\oplus m^- \; \colon\; B_{\Phi} \isom B_{\rho}.
$$
Then $m\tensor\Q$ maps $r_i\sprime$ to $[C_i\sprime]$ and $r_i\spprime$ to $[C_i\spprime]$.
Therefore we obtain an isometry
$m\colon M_f\isom M_{\rho}$.
Let $\primM_{f}$ be the even overlattice of $M_f$ corresponding to 
the even overlattice $\primM_{\rho}$ of $M_{\rho}$ via $m$,
so that $m\colon M_f\isom M_{\rho}$ extends to
$m\colon \primM_{f}\isom \primM_{\rho}$.
Then  $m$ induces an isomorphism
from the sequence~\eqref{eq:fseq} to the sequence~\eqref{eq:rhoseq}.
In particular, $m$ is a strong-realization isometry.
\end{proof}
\begin{theorem}\label{thm:geom2}
If an even overlattice $\primM_{f}$ of $M_f$
is strongly realized by an $\RDP$-Enriques surface $(Y, \rho)$,
then $\primM_{f}$
satisfies  $\condone$-$\condfour$.
\end{theorem} 
\begin{proof}
Let $m\colon \primM_{f}\isom \primM_{\rho}$
be a strong-realization isometry.
Since $\primM_{\rho}$ is a primitive sublattice of the primitive sublattice $\SX$
of the $K3$-lattice $H^2(X, \Z)\cong \LK$, 
the lattice $\primM_{f}\cong \primM_{\rho}$ satisfies $\condone$.
Since $ \pi^*\SY$ is primitive in $\SX$ and hence in $\primM_{\rho}$,
the lattice $\primM_{f}$ satisfies $\condtwo$.
Since the  isometry $m$ maps the lattice $N(\primM_{f})$ isomorphically  to a sublattice of $N_{Y}$,
Lemma~\ref{lem:Npi} implies that $\primM_{f}$ satisfies $\condthree$.
If $r\in \Roots(\primR_f)$ is 
$\primM_{f}$-liftable, then 
$m_Y(r)\in  \Roots(\primR_{\rho})$ is liftable,
because $m$ maps $\TTT(\primM_{f})$ to a subset of  $\TTT_{Y}$.
Hence the isometry $m_Y\colon \Lten\isom \SY$ induced by  $m$ induces the horizontal injection in the commutative diagram below:
$$
\renewcommand{\arraystretch}{1.4}
\begin{array}{ccc}
\shortset{r\in \Roots(\primR_f)}{\textrm{$r$ is $\primM_f$-liftable}} & \inj &\shortset{r\in \Roots(\primR_{\rho})}{\textrm{$r$ is liftable}} \\
\llap{\scriptsize by Lemma~\ref{lem:rootsRfliftable}\;\;} \raise -3pt \hbox{\rotatebox{90}{$\inj$}} 
&&  
\raise -2pt \hbox{\rotatebox{90}{$\cong$}} \rlap{\scriptsize \;\;\;by Proposition~\ref{prop:liftrho}} \\
\Roots(R_f) &\maprightspsb{\sim}{\textrm{by $m_Y$}} &\Roots(R_{\rho}).
\end{array}
$$
By Proposition~\ref{prop:liftrho} and 
Lemma~\ref{lem:rootsRfliftable},
we see that the upward injection in the left column of the diagram above is a bijection.
Hence  $\primM_{f}$ satisfies $\condfour$.
\end{proof}
\begin{theorem}\label{thm:geom3}
Let  $\primM_f$ be an even overlattice of $M_f$
satisfying  $\condone$-$\condfour$.
Then there exists 
an $\RDP$-Enriques surface $(Y, \rho)$
that strongly realizes 
 $\primM_f$. 
\end{theorem}
\begin{proof}
The main tool is the Torelli theorem 
for $K3$ surfaces and the surjectivity of the period map of  $K3$ surfaces
(see~\cite[Chapter VIII]{BHPVBook}).
By~$\condone$ for $\primM_f$ and the surjectivity of the period map, 
there exists a $K3$ surface $X$ equipped with a marking
$$
\mu\colon \primM_{f}\cong \SX.
$$
Our task is to construct an Enriques involution $\enrinvol$ on $X$,
and an $\RDP$-Enriques surface $\rho\colon Y:=X/\gen{\enrinvol}\to \Ybar$
that strongly realizes $\primM_f$. 
\par
Replacing  $\mu$ by $-\mu$ if necessary,
we can assume that $\mu$ induces an embedding
$$
\mu\circ \varpi^*|_{\PPP}  \;\; \colon\;\;  \Pten \inj \PPP_X.
$$
We put $\PPP(\Phi_f):=(\Phi_f)\sperp=([\Phi_f]\sperp\tensor\R)\cap \Pten$,
which is  a positive cone of the hyperbolic sublattice $[\Phi_f]\sperp=R_f\sperp=\primR_f\sperp$ of $\Lten$.
There exists a Vinberg chamber $\Delta$ of $\Lten$ such that $\PPP(\Phi_f)\cap \Delta$ 
contains a non-empty open subset of $\PPP(\Phi_f)$.
Moving $\Delta$ by an element of the subgroup $W(\Phi_f, \Lten)\subset \OG^+(\Lten)$ generated by the reflections 
$s_r$ with respect to the roots $r\in \Phi_f$,
we can assume that $\Delta$  satisfies the following:
\begin{equation}\label{eq:PhifDelta}
\intf{r, v}\ge 0 \quad \textrm{holds for all $r\in \Phi_f$ and $v\in \Delta$}.
\end{equation}
Let $r\sprime$ be a root of $\primM_f$.
By $\condthree$  for $\primM_f$,
the locus 
$$
(r\sprime)\sperp:=\set{x\in \Pten}{\intf{\varpi^* x, r\sprime}=0}
$$ 
is a hyperplane of $\Pten$.
Note that 
the family $\shortset{(r\sprime)\sperp}{r\sprime\in \Roots(\primM_f)}$ of hyperplanes of $\PPP_{10}$ 
is  locally finite.
Let $\eta\sprime$ be a general point of $\PPP(\Phi_f)\cap \Delta\cap (\Lten\tensor\Q)$, 
let $\UUU$ be a sufficiently small neighborhood of $\eta\sprime$ in $\Pten$,
and let $\alpha\sprime$ be a general point of $\UUU\cap \Delta\cap (\Lten\tensor\Q)$.
We have  a positive integer $c$ such that $\alpha:=c\alpha\sprime\in \Lten$ and $\eta:=c\eta\sprime\in \Lten$.
Then we have the following:
\begin{enumerate}[(i)]
\item  there exist no roots $r\sprime$ in $\primM_f$ such that $\varpi^*\alpha\in (r\sprime)\sperp$,
\item there exist no roots $r\sprime$ in $\primM_f$ such that $\intf{\varpi^*\alpha, r\sprime}>0$ and $\intf{\varpi^*\eta, r\sprime}<0$, and
\item the set of  $r\in \Roots(\Lten)$ satisfying $\intf{\eta, r}=0$ is equal to $\Roots(\primR_f)$.
\end{enumerate}
%
We then put
$$
a_X:=\mu\circ \varpi^*|_{\PPP} (\alpha)\in \PPP_X\cap \SX,
\quad
h_X:=\mu\circ \varpi^*|_{\PPP} (\eta)\in \PPP_X\cap \SX.
$$
By (i), we 
compose the marking $\mu$ by an element of $W(\SX)$
and assume that
$a_X$ is ample.
Then (ii) implies that
 $h_X$ is nef.
\par
By $\condtwo$, the sublattice  $\mu(\varpi^* \Lten)$ of $\SX$ is primitive in $\SX$
and hence is primitive in the even unimodular lattice  $H^2(X, \Z)$.
Let $K$ denote the orthogonal complement  of $\mu(\varpi^* \Lten)$ in $H^2(X, \Z)$.
Then we have an isomorphism
$$
\discg{K}\cong \discg{\mu(\varpi^* \Lten)} \cong \discg{\Lten(2)}\cong \F_2^{10}
$$
of discriminant groups by~\cite{Nikulin1979}. 
Since $\discg{K}$ is $2$-elementary,
the scalar multiplication by $-1$ on $K$ acts on $\discg{K}$ trivially.
By~\cite{Nikulin1979} again,
we obtain an isometry $\enrinvol\sprime$ of $H^2(X, \Z)$
 that  preserves each of 
 $\mu(\varpi^* \Lten)$ and $ K$, induces the identity on $\mu(\varpi^* \Lten)$,
 and induces the scalar multiplication by $-1$  on $K$.
 Since $\enrinvol\sprime$ acts on $H^0(X, \Omega_X^2)$ as the scalar multiplication by $-1$, 
 it preserves the Hodge structure of $H^2(X, \Z)$.
In particular,
the action of  $\enrinvol\sprime$
preserves $\SX$.
 Since $\enrinvol\sprime$ fixes the ample class $a_X$,
 the action of  $\enrinvol\sprime$ preserves the nef cone of $X$.
Hence the Torelli theorem implies that $\enrinvol\sprime$ comes from an involution $\enrinvol\colon X\to X$.
 By Proposition~\ref{prop:enrinvolNikulin}, we see that $\enrinvol$ is an Enriques involution.
 Let $\pi\colon X\to Y:=X/\gen{\enrinvol}$ be the quotient morphism.
 Since $\mu(\varpi^* \Lten)$ is the invariant sublattice of $\enrinvol\sprime$
 in $H^2(X, \Z)$,
we have $\mu(\varpi^* \Lten)= \pi^*\SY$.
 Therefore $\mu$ induces an isometry
 $$
 \mu_Y\colon \Lten\isom \SY.
 $$
  \par
 Note that  $a:=\mu_Y(\alpha)$
 is ample on $Y$ and $h:=\mu_Y(\eta)$ is nef on $Y$, because $\pi^*a=a_X$ and $\pi^*(h)=h_X$.
 Let $H$ be the nef divisor of $Y$ whose class is $h$.
 We construct a birational morphism $\rho_h\colon Y\to \Ybar$   by $H$
 as in Section~\ref{subsec:geomenr}.
It remains to show that $\mu_Y$ maps $\Phi_f$ to $\Phi_h:=\Phi_{\rho_h}$ bijectively.
Recall the definition~\eqref{eq:PihLh} of $\Pi^+_h$.
The property (iii) of $\eta$ above implies that
$\mu_Y$ maps $\Roots(\primR_f)$ to $\Pi^+_h\cup (-1)\Pi^+_h$ bijectively.
Hence $\mu_Y$ identifies $\primR_f$ and the sublattice $\gen{\Pi^+_h}$ of $\SY$ generated by $\Pi^+_h$.
Since $\mu$ is  an isometry from $\primM_f$ to $\SX$,
the map $\mu$ induces a bijection from $\TTT(\primM_f)$ to $\TTT_Y$.
Therefore
the isometry $\mu_Y$ induces a bijection from  the set of   $\primM_{f}$-liftable roots of $\primR_f$
to the set of  liftable root  of $\gen{\Pi^+_h}$.
Recall that $\alpha$ is in the interior of the Vinberg chamber $\Delta$.
By~$\condfour$ for $\primM_f$ and~\eqref{eq:PhifDelta},
the $\ADE$-basis of   the root sublattice 
generated by $\shortset{r\in \Roots(\primR_f)}{\textrm{$r$ is $\primM_{f}$-liftable}}$
associated with the linear form $\intf{-, \alpha}$ is $\Phi_f$.
On the other hand, 
Proposition~\ref{prop:Phih} implies that
the $\ADE$-basis of  the root sublattice 
generated by $\shortset{r\in \Roots(\gen{\Pi^+_h})}{\textrm{$r$ is liftable}}$
associated with $\intf{-, a}$
is $\Phi_{h}$.
Therefore $\mu_Y$  maps $\Phi_f$
to $\Phi_{h}$  bijectively.
\end{proof}
\begin{corollary}\label{cor:geom}
An embedding $f\colon \Phi\inj\Lten$
is geometrically realizable 
if and only if there exists an even overlattice $\primM_{f}$
of $M_f$ satisfying the conditions $\condone$-$\condfour$.
\qed
\end{corollary}
\begin{corollary}\label{cor:Usprime}
There exists a bijection between 
the set of strong equivalence classes of $\RDP$-Enriques surfaces
geometrically realizing $f\colon \Phi\inj\Lten$
and the set of orbits of the action of the group
$$
U(M_f):=\set{g\in \OG^+(M_f)}{{\varpi^* \Phi_f}^g=\varpi^*\Phi_f,\;\; {\varpi^*\Lten}^g=\varpi^*\Lten}.
$$
on the set of even overlattices of $M_f$ satisfying $\condone$-$\condfour$.
\qed
\end{corollary}
\begin{remark}\label{rem:sodoesMf}
Let $\primM_f$ and $\primM_f\sprime$
be even overlattices of $M_f$ such that $\primM_f\subset \primM_f\sprime$.
If $\primM_f\sprime$ satisfies $\condtwo$, $\condthree$, and  $\condfour$,
then so does $\primM_f$.
In particular,
if an even overlattice $\primM_f$ satisfies $\condtwo$, $\condthree$, and  $\condfour$,
then so does $M_f$.
\end{remark}
%
%
A finite generating set of the group $U(M_f)$ is calculated as follows.
First we define a homomorphism
$$
\Stab(\Phi_f, \Lten)\to \OG^+(M_f), \qquad g\mapsto \tilde{g}.
$$
Let $g$ be an element of $\Stab(\Phi_f, \Lten)$.
Then $g$ induces an automorphism of the $\ADE$-configuration $\Phi_f$.
Since $\Phi_f$ and $\Phi^-$ are canonically isomorphic by $r_i^+\mapsto r_i^-$,
we obtain an automorphism $g^-\in \Aut(\Phi^-)$ and hence
an isometry $g^-\in \OG(\gen{\Phi^-})$.
The action of $g\oplus g^-\in \OG(B_{\Phi})$
indices a permutation of $r_1\sprime, \dots, r_n\sprime\in B_{\Phi}\dual$, and hence 
preserves the even overlattice $M_f$ of $B_{\Phi}$.
Therefore we obtain $\tilde{g}\in \OG^+(M_f)$.
\par
Let $c$ be the number of connected components
of the Dynkin diagram of $\Phi^-$,
and let
$$
\gen{\Phi^-}(2)=\gen{\Phi^-_1}(2)\oplus \dots \oplus \gen{\Phi^-_c}(2)
$$
be the orthogonal direct-sum decomposition of $\gen{\Phi^-}(2)$
according to the connected components
of the Dynkin diagram of $\Phi^-$.
For $k=1, \dots, c$, 
let $u_k\sprime\in \OG(\gen{\Phi^-}(2))$ denote   the isometry that
acts on the direct-summand  $\gen{\Phi^-_j}(2)$ as the identity for $j\ne k$
and as the multiplication by $-1$ for $j=k$.
We then define  $u_k\in \OG(B_{\Phi})$
to be the direct sum of $\id_{\Lten(2)}$ and $u_k\sprime$.
Then $u_k$ acts on each of the subsets $\{r_j\sprime, r_j\spprime\}$
of $B_{\Phi}\dual$, 
and hence we can regard $u_k$ as an element of $\OG^+(M_{f})$.
\begin{proposition}\label{prop:U}
Suppose that there exists an  $\RDP$-Enriques surfaces
geometrically realizing $f\colon \Phi\inj\Lten$.
Then the group $U(M_f)$ is generated by 
the image of the homomorphism $\Stab(\Phi_f, \Lten)\to \OG^+(M_f)$ above 
and the   isometries $u_1, \dots, u_c$.
\end{proposition}
\begin{proof}
Let $U\sprime\subset \OG^+(M_f)$ be the group 
generated by 
the image of the homomorphism $\Stab(\Phi_f, \Lten)\to \OG^+(M_f)$
and the   isometries $u_1, \dots, u_c$.
By construction,
we have $U\sprime\subset U(M_f)$.
We prove $U\sprime\supset U(M_f)$.
Since $\varpi^*\Lten=\Lten(2)$, 
we have a natural restriction map
$$
\xi\colon  U(M_f)\to \OG^+(\varpi^*\Lten)=\OG^+(\Lten), \quad g\mapsto g |_{\varpi^*\Lten}.
$$
The image of $\xi$ is contained in $\Stab(\Phi_f, \Lten)$ by definition, 
and $\xi$ has a section over $\Stab(\Phi_f, \Lten)\subset \OG^+(\Lten)$.
Therefore it suffices  to show that 
an arbitrary element $g$ of $\Ker \xi$ belongs to the   subgroup 
generated by $u_1, \dots, u_c$.
By the assumption and
Remark~\ref{rem:sodoesMf},  the orthogonal complement $N(M_f)$ of $\varpi^*\Lten$ in $M_f$
contains no roots.
Hence 
the same argument as in the proof of Lemma~\ref{lem:twolifts} implies that 
each $r_i^+\in \Phi_f$ has exactly two $N(M_f)$-lifts $r_i\sprime, r_i\spprime$.
Since $r_i^{+g}=r_i^+$,
we have $\{r_i\sprime, r_i\spprime\}^g=\{r_i\sprime, r_i\spprime\}$.
Since $\varphi (r_i^-)=r_i\sprime-r_i\spprime$,
we have $\varphi (r_i^-)^g=\pm \varphi (r_i^-)$.
By~\eqref{eq:connectrirj},
if $\intf{r_i, r_j}\ne 0$,
then $\varphi (r_i^-)^g$ determines $\varphi (r_j^-)^g$.
Therefore the action of $g$ on $\gen{\Phi^-}(2)$
is equal to the action of an element of $\gen{u_1, \dots, u_c}$ on $\gen{\Phi^-}(2)$.
\end{proof}
\subsection{Complete list of $\RDP$-Enriques surfaces}\label{subsec:list}
In Section~\ref{sec:Lten},
we have calculated the complete list of equivalence classes of
embeddings $f\colon \Phi\inj \Lten$.
From each equivalence class,
we choose a representative $f\colon \Phi\inj \Lten$,
calculate a finite generating set of the group $\Stab(\Phi_f, \Lten)$
by the method given in Section~\ref{subsec:stabPhiL},
and calculate the lattice $M_f$ and a finite  generating set of 
the group $U(M_f)\subset \OG^+(M_f)$.
Then we calculate the finite set $\LLL\sprime(M_f)$ of even overlattices of $M_f$
that satisfy $\condtwo$,  $\condthree$, $\condfour$
by Proposition~\ref{prop:overlattices}.
\begin{remark}\label{rem:enlarge}
We enumerate even overlattices of $M_f$
by enlarging successively the corresponding totally isotropic subgroups
of the discriminant form $\discf{M_f}$ of $M_f$.
By Remark~\ref{rem:sodoesMf},
 if $\primM_f$ fails to satisfy $\condtwo$, $\condthree$, or $\condfour$,
then we do not have to enlarge the  totally isotropic subgroup  $\primM_f/M_f$ any more.
\end{remark}
We then  decompose $\LLL\sprime(M_f)$ into the union of the orbits 
under the action of  $U(M_f)$. 
Note that the image of
$U(M_f)$ by the natural homomorphism $\OG(M_f)\to \OG(\discf{M_f})$
is of course finite,  and is calculated from the finite generating set of $U(M_f)$.
For each orbit $o$,
we choose a representative element $\primM_f\in o$,
and check whether $\primM_f$ satisfies $\condone$ or not 
by Proposition~\ref{prop:genus}
and Remark~\ref{rem:genus}.
If $\primM_f$ satisfies $\condone$,
then $o$ corresponds to a strong equivalence class.
Thus we  obtain the complete list of strong equivalence classes of $\RDP$-Enriques surfaces.
The result is given in Table~\ref{table:main1},
and Theorem~\ref{thm:strongmain} is proved.

%
%
%
\bibliographystyle{plain}

\end{document}

%% file: mymacros.tex
\newcommand{\erase}[1]{}

\newtheorem{theorem}{Theorem}[section]
\newtheorem{lemma}[theorem]{Lemma}
\newtheorem{proposition}[theorem]{Proposition}
\newtheorem{corollary}[theorem]{Corollary}

\newtheorem{_algorithm}[theorem]{Algorithm}
\newenvironment{algorithm}{\begin{_algorithm}\rm}{\hfill \rule{3pt}{6pt}\end{_algorithm}}

\newtheorem{_procedure}[theorem]{Procedure}

\newtheorem{_definition}[theorem]{Definition}
\newenvironment{definition}{\begin{_definition}\rm}{\end{_definition}}

\newtheorem{_remark}[theorem]{\it Remark}
\newenvironment{remark}{\begin{_remark}\rm}{\end{_remark}}

\newtheorem{_example}[theorem]{Example}
\newenvironment{example}{\begin{_example}\rm}{\end{_example}}

\newtheorem{_assumption}[theorem]{Assumption}

\newtheorem{_construction}[theorem]{Construction}

\newtheorem{_claim}[theorem]{Claim}

\newtheorem{_conjecture}[theorem]{Conjecture}

\numberwithin{equation}{section}
\numberwithin{table}{section}
\numberwithin{figure}{section}
\renewcommand{\qed}{\hfill {$\Box$}}

\newcommand{\C}{\mathord{\mathbb C}}

\newcommand{\F}{\mathord{\mathbb F}}

\renewcommand{\P}{\mathord{\mathbb  P}}
\newcommand{\Q}{\mathord{\mathbb  Q}}
\newcommand{\R}{\mathord{\mathbb R}}

\newcommand{\Z}{\mathord{\mathbb Z}}

\newcommand{\AAA}{\mathord{\mathcal A}}
\newcommand{\BBB}{\mathord{\mathcal B}}
\newcommand{\CCC}{\mathord{\mathcal C}}
\newcommand{\DDD}{\mathord{\mathcal D}}

\newcommand{\GGG}{\mathord{\mathcal G}}
\newcommand{\HHH}{\mathord{\mathcal H}}
\newcommand{\III}{\mathord{\mathcal I}}

\newcommand{\LLL}{\mathord{\mathcal L}}

\newcommand{\NNN}{\mathord{\mathcal N}}
\newcommand{\OOO}{\mathord{\mathcal O}}
\newcommand{\PPP}{\mathord{\mathcal P}}

\newcommand{\RRR}{\mathord{\mathcal R}}
\newcommand{\SSS}{\mathord{\mathcal S}}
\newcommand{\TTT}{\mathord{\mathcal T}}
\newcommand{\UUU}{\mathord{\mathcal U}}
\newcommand{\VVV}{\mathord{\mathcal V}}
\newcommand{\WWW}{\mathord{\mathcal W}}

\newcommand{\SSSS}{\mathord{\mathfrak S}}

\newcommand{\maprightsp}[1]{\; \smash{\mathop{\; \to \; }\limits\sp{#1}}\; }

\newcommand{\maprightspsb}[2]{\; \smash{\mathop{\; \longrightarrow \; }\limits\sp{#1}\limits\sb{#2}}\; }

\newcommand{\maprightinjsp}[1]{\; \smash{\mathop{\; \inj \; }\limits\sp{#1}}\; }

\newcommand{\mapdown}{\phantom{\Big\downarrow}\hskip -8pt \downarrow}
\newcommand{\mapdownright}[1]{\mapdown\rlap{$\vcenter{\hbox{$\scriptstyle#1$}}$}}

\newcommand{\mapdownsurj}{
\hbox{$\bigm\downarrow$}
\llap{\hbox{\raise 2pt\hbox{$\bigm\downarrow$}}}%
\vstrechmapdown
}

\newcommand{\mapupsurj}{
\hbox{$\bigm\uparrow$}
\llap{\hbox{\raise 2pt\hbox{$\bigm\uparrow$}}}%
\vstrechmapup
}

\newcommand{\inj}{\hookrightarrow}

\newcommand{\isom}{\xrightarrow{\hskip -1.5pt \raise -.5pt\hbox{\tiny $\sim$}}}

\newcommand{\set}[2]{\{\; {#1} \; \mid \; {#2} \;  \}}
\newcommand{\shortset}[2]{\{ {#1} \,|\, {#2}   \}}

\newcommand{\gen}[1]{\langle {#1}  \rangle}

\newcommand{\tensor}{\otimes}

\newcommand{\sprime}{\sp\prime}

\newcommand{\spar}[1]{\sp{(#1)}}
\newcommand{\spprime}{\sp{\prime\prime}}
\newcommand{\spcirc}{\sp{\mathord{\circ}}}

\newcommand{\sperp}{\sp{\perp}}

\newcommand{\dual}{\sp{\vee}}

\newcommand{\inv}{\sp{-1}}

\newcommand{\Hom}{\mathord{\mathrm{Hom}}}

\newcommand{\OG}{\mathord{\mathrm{O}}}

\newcommand{\id}{\mathord{\mathrm{id}}}

\newcommand{\Ker}{\operatorname{\mathrm{Ker}}\nolimits}

\newcommand{\Image}{\operatorname{\mathrm{Im}}\nolimits}
\newcommand{\Aut}{\mathord{\mathrm{Aut}}}

\newcommand{\closure}[1]{\overline{#1}}

\newcommand{\mystruth}[1]{\phantom{\hbox{\vrule height #1}}}
\newcommand{\mystrutd}[1]{\phantom{\hbox{\vrule depth #1}}}

\newcommand{\pione}{\pi_1}

%% file: mymacrosEnrRs.tex
\newcommand{\intf}[1]{\langle #1\rangle}
\newcommand{\intfvoid}{\intf{\phantom{\cdot}, \phantom{\cdot}}}

\newcommand{\RDP}{\mathord{\rm RDP}}
\newcommand{\ADE}{\mathord{\rm ADE}}
\newcommand{\Stab}{\mathord{\rm Stab}}
\newcommand{\Isom}{\mathord{\rm Isom}}
\newcommand{\Nef}{\mathord{\rm Nef}}

\newcommand{\Ybar}{\overline{Y}}
\newcommand{\SY}{S_Y}
\newcommand{\SX}{S_X}

\newcommand{\Lten}{L_{10}}
\newcommand{\Pten}{\PPP_{10}}

\newcommand{\discg}[1]{A_{#1}}
\newcommand{\discf}[1]{q_{#1}}
\newcommand{\refl}[1]{s_{#1}}

\newcommand{\primM}{\overline{M}}
\newcommand{\primR}{\overline{R}}

\newcommand{\Emb}{\mathord{\rm Emb}}
\newcommand{\baremb}{\overline{\mathord{\rm emb}}}

\newcommand{\OL}{\LLL}

\newcommand{\Sigmas}{\SSS}

\newcommand{\VVVS}{\VVV_{\Sigma}}

\newcommand{\sbpar}[1]{\sb{(#1)}}

\newcommand{\Delmin}{\Delta_{\min}}

\newcommand{\res}{\mathord{\rm res}}

\newcommand{\enrinvol}{\varepsilon}

\newcommand{\Roots}{\mathord{\rm Roots}}

\newcommand{\LK}{L_{\mathord{\rm K3}}}
\newcommand{\condone}{\mathord{\rm (C1)}}
\newcommand{\condtwo}{\mathord{\rm (C2)}}
\newcommand{\condthree}{\mathord{\rm (C3)}}
\newcommand{\condfour}{\mathord{\rm (C4)}}

%% file: MainTable.tex
%
%
\begin{table}
$$
{\small
\hskip -1cm
\begin{array}{llll}
\textrm{No.} \phantom{aa} & \tau(\Phi_f) & \tau(\primR_f) &  Q_{(Y, \rho)}  \mystrutd{6pt}\\ 
\hline
1&A_{1}&\Phi&0 \mystruth{12pt}\\
2&2A_{1}&\Phi&0\\
3&A_{2}&\Phi&0\\
4&3A_{1}&\Phi&0\\
5&A_{1}+A_{2}&\Phi&0\\
6&A_{3}&\Phi&0\\
7&4A_{1}&\Phi&0\\
8&4A_{1}&D_{4}&0, 2\\
9&2A_{1}+A_{2}&\Phi&0\\
10&A_{1}+A_{3}&\Phi&0\\
11&2A_{2}&\Phi&0\\
12&A_{4}&\Phi&0\\
13&D_{4}&\Phi&0\\
14&5A_{1}&\Phi&0\\
15&5A_{1}&A_{1}+D_{4}&0, 2\\
16&3A_{1}+A_{2}&\Phi&0\\
17&2A_{1}+A_{3}&\Phi&0\\
18&2A_{1}+A_{3}&D_{5}&0, 2\\
19&A_{1}+2A_{2}&\Phi&0\\
20&A_{1}+A_{4}&\Phi&0\\
21&A_{1}+D_{4}&\Phi&0\\
22&A_{2}+A_{3}&\Phi&0\\
23&A_{5}&\Phi&0\\
24&D_{5}&\Phi&0\\
25&6A_{1}&2A_{1}+D_{4}&0, 2\\
26&6A_{1}&D_{6}&2, 22\\
27&4A_{1}+A_{2}&\Phi&0\\
28&4A_{1}+A_{2}&A_{2}+D_{4}&0, 2\\
29&3A_{1}+A_{3}&\Phi&0\\
30&3A_{1}+A_{3}&A_{1}+D_{5}&0, 2\\
31&2A_{1}+2A_{2}&\Phi&0\\
32&2A_{1}+A_{4}&\Phi&0\\
33&2A_{1}+D_{4}&\Phi&0\\
34&2A_{1}+D_{4}&D_{6}&0, 2\\
35&A_{1}+A_{2}+A_{3}&\Phi&0\\
36&A_{1}+A_{5}&\Phi&0\\
37&A_{1}+A_{5}&E_{6}&0, 2\\
38&A_{1}+D_{5}&\Phi&0\\
39&3A_{2}&\Phi&0, 3\\
40&3A_{2}&E_{6}&0\\
41&A_{2}+A_{4}&\Phi&0\\
42&A_{2}+D_{4}&\Phi&0\\
43&2A_{3}&\Phi&0\\
44&2A_{3}&D_{6}&0, 2, 2\\
45&A_{6}&\Phi&0\\
46&D_{6}&\Phi&0\\
47&E_{6}&\Phi&0\\
48&7A_{1}&A_{1}+D_{6}&2\\
49&7A_{1}&E_{7}&22\\
50&5A_{1}+A_{2}&A_{1}+A_{2}+D_{4}&0
\end{array}
\hskip .5cm
\begin{array}{llll}
\textrm{No.}  \phantom{aa} &\tau(\Phi_f) & \tau(\primR_f) &  Q_{(Y, \rho)}   \mystrutd{6pt}\\ 
\hline
51&4A_{1}+A_{3}&2A_{1}+D_{5}&0, 2 \mystruth{12pt}\\
52&4A_{1}+A_{3}&A_{3}+D_{4}&0, 2\\
53&4A_{1}+A_{3}&D_{7}&2, 2, 2, 22\\
54&3A_{1}+2A_{2}&\Phi&0\\
55&3A_{1}+A_{4}&\Phi&0\\
56&3A_{1}+D_{4}&A_{1}+D_{6}&0, 2\\
57&3A_{1}+D_{4}&E_{7}&2, 22\\
58&2A_{1}+A_{2}+A_{3}&\Phi&0\\
59&2A_{1}+A_{2}+A_{3}&A_{2}+D_{5}&0, 2\\
60&2A_{1}+A_{5}&\Phi&0\\
61&2A_{1}+A_{5}&A_{1}+E_{6}&0, 2\\
62&2A_{1}+D_{5}&\Phi&0\\
63&2A_{1}+D_{5}&D_{7}&0, 2\\
64&A_{1}+3A_{2}&\Phi&0, 3\\
65&A_{1}+3A_{2}&A_{1}+E_{6}&0\\
66&A_{1}+A_{2}+A_{4}&\Phi&0\\
67&A_{1}+A_{2}+D_{4}&\Phi&0\\
68&A_{1}+2A_{3}&\Phi&0\\
69&A_{1}+2A_{3}&A_{1}+D_{6}&0, 2, 2, 4\\
70&A_{1}+2A_{3}&E_{7}&0, 2\\
71&A_{1}+A_{6}&\Phi&0\\
72&A_{1}+D_{6}&\Phi&0\\
73&A_{1}+D_{6}&E_{7}&0, 2\\
74&A_{1}+E_{6}&\Phi&0\\
75&2A_{2}+A_{3}&\Phi&0\\
76&A_{2}+A_{5}&\Phi&0, 3\\
77&A_{2}+A_{5}&E_{7}&0\\
78&A_{2}+D_{5}&\Phi&0\\
79&A_{3}+A_{4}&\Phi&0\\
80&A_{3}+D_{4}&\Phi&0\\
81&A_{3}+D_{4}&D_{7}&0, 2, 2\\
82&A_{7}&\Phi&0\\
83&A_{7}&E_{7}&0, 2, 2\\
84&D_{7}&\Phi&0\\
85&E_{7}&\Phi&0\\
86&8A_{1}&A_{1}+E_{7}&-\\
87&8A_{1}&D_{8}&22\\
88&8A_{1}&E_{8}&222\\
89&6A_{1}+A_{2}&A_{2}+D_{6}&-\\
90&5A_{1}+A_{3}&A_{1}+D_{7}&2\\
91&4A_{1}+A_{4}&A_{4}+D_{4}&0\\
92&4A_{1}+D_{4}&A_{1}+E_{7}&2\\
93&4A_{1}+D_{4}&D_{8}&2, 2\\
94&4A_{1}+D_{4}&E_{8}&22, 22\\
95&3A_{1}+A_{2}+A_{3}&A_{1}+A_{2}+D_{5}&0\\
96&3A_{1}+A_{5}&2A_{1}+E_{6}&0, 2\\
97&3A_{1}+D_{5}&A_{1}+D_{7}&0\\
98&2A_{1}+3A_{2}&2A_{1}+E_{6}&0\\
99&2A_{1}+A_{2}+A_{4}&\Phi&0\\
100&2A_{1}+A_{2}+D_{4}&A_{2}+D_{6}&0
\end{array}
}
$$
\vskip .5cm
\caption{$\ADE$-configurations of roots in $\Lten$ (continues)}
\label{table:main1}
\end{table}
\setcounter{table}{0}
\begin{table}
$$
\hskip -1cm
{\small
\begin{array}{llll}
\textrm{No.}  \phantom{aa} &\tau(\Phi_f) & \tau(\primR_f)&  Q_{(Y, \rho)}  \mystrutd{6pt} \\ 
\hline
101&2A_{1}+2A_{3}&A_{1}+E_{7}&0, 2, 4 \mystruth{12pt}\\
102&2A_{1}+2A_{3}&A_{3}+D_{5}&0, 2\\
103&2A_{1}+2A_{3}&D_{8}&2, 2, 2, 22, 22, 4, 42\\
104&2A_{1}+2A_{3}&E_{8}&2, 2, 2, 22\\
105&2A_{1}+A_{6}&\Phi&0\\
106&2A_{1}+D_{6}&A_{1}+E_{7}&0, 2\\
107&2A_{1}+D_{6}&D_{8}&0, 2\\
108&2A_{1}+D_{6}&E_{8}&2, 2, 2, 22\\
109&2A_{1}+E_{6}&\Phi&0\\
110&A_{1}+2A_{2}+A_{3}&\Phi&0\\
111&A_{1}+A_{2}+A_{5}&\Phi&0, 3\\
112&A_{1}+A_{2}+A_{5}&A_{1}+E_{7}&0\\
113&A_{1}+A_{2}+A_{5}&A_{2}+E_{6}&0, 2, 3, 6\\
114&A_{1}+A_{2}+A_{5}&E_{8}&0, 2\\
115&A_{1}+A_{2}+D_{5}&\Phi&0\\
116&A_{1}+A_{3}+A_{4}&\Phi&0\\
117&A_{1}+A_{3}+D_{4}&A_{1}+D_{7}&0, 2\\
118&A_{1}+A_{7}&\Phi&0\\
119&A_{1}+A_{7}&A_{1}+E_{7}&0, 2, 2, 4, 4\\
120&A_{1}+A_{7}&E_{8}&0, 2, 2\\
121&A_{1}+D_{7}&\Phi&0\\
122&A_{1}+E_{7}&\Phi&0\\
123&A_{1}+E_{7}&E_{8}&0, 2\\
124&4A_{2}&A_{2}+E_{6}&0, 3\\
125&4A_{2}&E_{8}&0\\
126&2A_{2}+A_{4}&\Phi&0\\
127&A_{2}+2A_{3}&A_{2}+D_{6}&0, 2\\
128&A_{2}+A_{6}&\Phi&0\\
129&A_{2}+D_{6}&\Phi&0\\
130&A_{2}+E_{6}&\Phi&0, 3\\
131&A_{2}+E_{6}&E_{8}&0\\
132&A_{3}+A_{5}&\Phi&0\\
133&A_{3}+D_{5}&\Phi&0\\
134&A_{3}+D_{5}&D_{8}&0, 2, 2, 4\\
135&A_{3}+D_{5}&E_{8}&0, 2\\
136&2A_{4}&\Phi&0, 5\\
137&2A_{4}&E_{8}&0\\
138&A_{4}+D_{4}&\Phi&0\\
139&A_{8}&\Phi&0, 3\\
140&A_{8}&E_{8}&0\\
141&2D_{4}&D_{8}&0, 2\\
142&2D_{4}&E_{8}&2, 2, 22
\end{array}
\hskip .5cm
\begin{array}{llll}
\textrm{No.} \phantom{aa} & \tau(\Phi_f) & \tau(\primR_f) &  Q_{(Y, \rho)} \mystrutd{6pt}\\ 
\hline
143&D_{8}&\Phi&0 \mystruth{12pt}\\
144&D_{8}&E_{8}&0, 2, 2\\
145&E_{8}&\Phi&0\\
146&9A_{1}&A_{1}+E_{8}&-\\
147&7A_{1}+A_{2}&A_{2}+E_{7}&-\\
148&6A_{1}+A_{3}&D_{9}&-\\
149&5A_{1}+D_{4}&A_{1}+E_{8}&-\\
150&4A_{1}+D_{5}&D_{9}&-\\
151&3A_{1}+A_{2}+D_{4}&A_{2}+E_{7}&-\\
152&3A_{1}+2A_{3}&A_{1}+E_{8}&2\\
153&3A_{1}+D_{6}&A_{1}+E_{8}&2\\
154&2A_{1}+A_{2}+A_{5}&A_{1}+E_{8}&0\\
155&2A_{1}+A_{3}+A_{4}&A_{4}+D_{5}&0\\
156&2A_{1}+A_{3}+D_{4}&D_{9}&2, 2\\
157&2A_{1}+A_{7}&A_{1}+E_{8}&0, 2, 4\\
158&2A_{1}+D_{7}&D_{9}&0\\
159&2A_{1}+E_{7}&A_{1}+E_{8}&0\\
160&A_{1}+4A_{2}&A_{1}+E_{8}&-\\
161&A_{1}+A_{2}+2A_{3}&A_{2}+E_{7}&0\\
162&A_{1}+A_{2}+A_{6}&\Phi&0\\
163&A_{1}+A_{2}+D_{6}&A_{2}+E_{7}&0\\
164&A_{1}+A_{2}+E_{6}&A_{1}+E_{8}&0\\
165&A_{1}+A_{3}+A_{5}&A_{3}+E_{6}&0, 2\\
166&A_{1}+A_{3}+D_{5}&A_{1}+E_{8}&0\\
167&A_{1}+2A_{4}&A_{1}+E_{8}&0\\
168&A_{1}+A_{8}&\Phi&0, 3\\
169&A_{1}+A_{8}&A_{1}+E_{8}&0\\
170&A_{1}+2D_{4}&A_{1}+E_{8}&2\\
171&A_{1}+D_{8}&A_{1}+E_{8}&0, 2\\
172&A_{1}+E_{8}&\Phi&0\\
173&3A_{2}+A_{3}&A_{3}+E_{6}&0\\
174&2A_{2}+A_{5}&A_{2}+E_{7}&0, 3\\
175&A_{2}+A_{7}&A_{2}+E_{7}&0, 2\\
176&A_{2}+E_{7}&\Phi&0\\
177&3A_{3}&D_{9}&2, 22\\
178&A_{3}+D_{6}&D_{9}&0, 2\\
179&A_{3}+E_{6}&\Phi&0\\
180&A_{4}+A_{5}&\Phi&0\\
181&A_{4}+D_{5}&\Phi&0\\
182&A_{9}&\Phi&0\\
183&D_{4}+D_{5}&D_{9}&0\\
184&D_{9}&\Phi&0\\
\end{array}
}
$$
\vskip .5cm
\caption{$\ADE$-configurations of roots in $\Lten$ (continued)}
\label{table:main2}
\end{table}